\documentclass[11pt, a4paper]{article}
\usepackage{mathrsfs}
\usepackage{amsmath}
\usepackage{amssymb}
\usepackage{float}
\usepackage[amsmath,thmmarks,hyperref]{ntheorem}
\usepackage[all]{xy}
\usepackage{stmaryrd}
\usepackage[bookmarksnumbered]{hyperref}
\usepackage{enumitem}
\usepackage{eucal}
\usepackage{ulem}
\usepackage{todonotes}
\usepackage{rotating}
\usepackage{appendix}
\usepackage{parskip}
\usepackage[textwidth=16cm, textheight=24cm]{geometry}
\usepackage{pdflscape}
\usepackage{setspace}
\usepackage{mathabx}

\let\OLDthebibliography\thebibliography
\renewcommand\thebibliography[1]{
 \OLDthebibliography{#1}
 \setlength{\parskip}{1.4pt}
 \setlength{\itemsep}{0pt plus 0.2ex}
}

\DeclareMathOperator{\cross}{cross}
\DeclareMathOperator{\nhomog}{\textit{n}--homog}
\DeclareMathOperator{\nexs}{\textit{n}--exs}
\DeclareMathOperator{\stable}{stable}
\DeclareMathOperator{\proj}{proj}
\DeclareMathOperator{\ml}{ml}
\DeclareMathOperator{\Sp}{Sp}
\DeclareMathOperator{\topspace}{Top}
\DeclareMathOperator{\diff}{diff}
\DeclareMathOperator{\hodiff}{hodiff}
\DeclareMathOperator{\skel}{sk}
\DeclareMathOperator{\cref}{cr}
\DeclareMathOperator{\hocref}{hocr}

\DeclareMathOperator{\nat}{Nat}
\DeclareMathOperator{\mapdiag}{map--diag}
\DeclareMathOperator{\obdiag}{ob--diag}
\DeclareMathOperator{\symfun}{Sym--Fun}

\newcommand{\devcat}{\Sigma_n \ltimes (\wcal_n \Top)}
\newcommand{\orthdevcat}{O(n) \ltimes (\jcal_n \Top)}

\newcommand{\lca}{\circlearrowleft}

\newcommand{\bigsmashprod}[2]{\underset{#1}{\overset{#2}{\bigwedge}}}
\newcommand{\Top}{{\topspace}}
\newcommand{\sphspec}{{\mathbb{S}}}

\newcommand{\mcal}{{ \mathcal{M} }}
\newcommand{\pcal}{{ \mathcal{P} }}
\newcommand{\ccal}{{ \mathcal{C} }}
\newcommand{\dcal}{{ \mathcal{D} }}
\newcommand{\scal}{{ \mathcal{S} }}
\newcommand{\ecal}{{ \mathcal{E} }}
\newcommand{\wcal}{{ \mathcal{W} }}
\newcommand{\jcal}{{ \mathcal{J} }}

\newcommand{\smashprod}{\wedge}
\newcommand{\co}{\colon \!}

\newcommand{\fibrep}{{ \widehat{r} }}

\DeclareMathOperator{\id}{Id}

\DeclareMathOperator{\holim}{holim}
\DeclareMathOperator{\hocolim}{hocolim}
\DeclareMathOperator{\colim}{colim}

\newtheorem{theorem}{Theorem}[section]
\newtheorem{proposition}[theorem]{Proposition}
\newtheorem{corollary}[theorem]{Corollary}
\newtheorem{lemma}[theorem]{Lemma}

\newtheorem{definition}[theorem]{Definition}

\newtheorem{ex}[theorem]{Example}

\newtheorem{remark}[theorem]{Remark}

\theoremsymbol{\ensuremath{\rule{0.6em}{0.6em}}}
\theoremheaderfont{\bfseries}
\theorembodyfont{\normalfont}
\theoremseparator{}
\newtheorem*{proof}{Proof.}

\newcommand{\ra}{\rightarrow}
\newcommand{\lra}{\longrightarrow}
\newcommand{\Fun}{\mathrm{Fun}}
\newcommand{\Nat}{\mathrm{Nat}}

\DeclareMathOperator{\Hom}{Hom}

\newcommand{\join}{\Asterisk}
\newcommand{\smsh}{\wedge}

\newcommand{\x}{\times}

\newcommand{\T}{\mathrm{T}}
\renewcommand{\P}{\mathrm{P}}
\newcommand{\Z}{\mathbb{Z}}

\title{Capturing Goodwillie's Derivative}
\author{David Barnes \and Rosona Eldred}

\begin{document}
\maketitle
\normalem
\begin{abstract}
\noindent Recent work of Biedermann and R\"ondigs
has translated Goodwillie's calculus of functors
into the language of model categories.
Their work focuses on symmetric multilinear functors
and the derivative
appears only briefly.
In this paper we focus on understanding the derivative as a right Quillen functor to a new model category.
This is directly analogous to the behaviour of Weiss's derivative
in orthogonal calculus.
The immediate advantage of this new category is
that we obtain a streamlined and more informative proof that
the $n$--homogeneous functors are classified by
spectra with a $\Sigma_n$-action.
In a later paper we will use this new model category
to give a formal comparison between the orthogonal calculus and
Goodwillie's calculus of functors.
\end{abstract}

\begin{spacing}{0.7}
\tableofcontents
\end{spacing}

\section{Introduction}
Goodwillie's calculus of homotopy functors is a highly successful method of studying
equivalence-preserving functors, often with source and target either spaces or spectra. The original development is given in the three papers by Goodwillie \cite{gw90, gw91, goodcalc3},  motivated by the study of Waldhausen's algebraic $K$-theory of a space. A family of related theories grew out of this work;  our focus is on
the homotopy functor calculus and (to a lesser extent in this paper) the orthogonal calculus of Weiss \cite{weiss95}. The orthogonal calculus was developed to study functors from real inner-product spaces to topological spaces, such as $BO(V)$ and $TOP(V)$.

The model categorical foundations for the homotopy functor calculus and the orthogonal calculus have been established; see
Biedermann-Chorny-R\"ondigs \cite{BCR07}, Biedermann-R\"ondigs \cite{BRgoodwillie} and Barnes-Oman\cite{barnesoman13}.
However, we have found these to be incompatible.
Most notably, the symmetric multilinear
functors of Goodwillie appear to have no analogue in the theory of Weiss.
In this paper, we re-work the classification results of Goodwillie to make it
resemble that of the orthogonal calculus. In a subsequent paper \cite{barneseldred15}, we will use this
similarity to give a formal comparison between the orthogonal calculus and
Goodwillie's calculus of functors.

This re-working marks a substantial difference from the existing literature on model structures
(or infinity categories) for Goodwillie calculus, as they follow the
pattern of Goodwillie's work in a variety of different contexts (see also Pereira \cite{pereira13} or Lurie \cite{luriehigher}).
Our setup takes a more equivariant perspective and has the advantage of using one less adjunction and fewer categories than that of \cite{BRgoodwillie} and \cite{goodcalc3}.
In detail, we construct a new category
($\devcat$, Section \ref{subsec:devcat}) which will be the target
of an altered notion of the derivative over a point
($\diff_n$, Section \ref{subsec:diffnquillen}).
This approach simplifies the classification of homogeneous functors
in terms of spectra with $\Sigma_n$-action, whilst retaining Goodwillie's
original classification at the level of homotopy categories, see
Theorem \ref{thm:correctdervied}. It also provides a new characterisation of the
$n$-homogeneous equivalences, see Lemma \ref{lem:stability}
and clarifies some important calculations, see
Examples \ref{ex:deriv1} and \ref{ex:deriv2}.

\subsection{Recent History and Context}
What the family of functor calculi have most in common is that they associate, to an equivalence-preserving functor $F$, a tower (the Taylor tower of $F$) of functors

\[
\xymatrix{
& D_n F \ar[d] & D_{n-1} F \ar[d] & & D_1 F \ar[d] \\
\cdots \ar[r]&\P_n F \ar[r] & \P_{n-1} F \ar[r] &\cdots \ar[r] & \P_1 F \ar[r] & P_0 F\\
}
\]
where the $\P_nF$ have a kind of $n$-polynomial property, and for nice functors, the inverse limit of the tower, denoted $\P_\infty F$, is equivalent to $F$. The layers of the tower, $D_nF$, are then analogous to purely-$n$-polynomial functors -- called $n$-homogeneous.
Figure \ref{fig:homogclass} represents Goodwillie's classification of (finitary) $n$-homogeneous functors in terms of spectra with $\Sigma_n$-action. This classification is phrased in
terms of three equivalences of homotopy categories.
\begin{figure}[H]
\scalebox{.85}{$
\xymatrix{
\text{Ho}(n\text{-homog-Fun}(C, \Top)) \ar@<5pt>[r]^{\cite[\S 2]{goodcalc3}} & \text{Ho}(n\text{-homog-Fun}(C, \text{Sp})) \ar@<5pt>[r]^{\cite[\text{Thm }3.5]{goodcalc3}}  \ar[l] & \text{Ho}(\text{Symm-Fun}(C^n, \text{Sp})_{ml})  \ar@<5pt>[r]^-{\cite[\S 5]{goodcalc3}}  \ar[l] &\text{Ho}( \Sigma_n \circlearrowleft \text{Sp}) \ar[l]
}$}
\caption{Goodwillie's classification}
\label{fig:homogclass}
\end{figure}
Here, ``Ho" indicates that we are working with homotopy categories,
$n\text{-homog-Fun}(A,B)$ is the category of $n$-homogeneous functors from $A$ to $B$,
$C$ is either spectra (Sp) or spaces (Top) and $ \Sigma_n \circlearrowleft \text{Sp}$ denotes (Bousfield-Friedlander) spectra with an action of  $\Sigma_n$.
The category $(\text{Symm-Fun}(C^n, \text{Sp})_{ml})$ consists of
symmetric multi-linear functors of $n$-inputs:
those $F$ with $F(X_1, \ldots, X_n) \cong F(X_{\sigma(1)}, \ldots, X_{\sigma(n)})$
for $\sigma \in \Sigma_n$ and which are degree 1-polynomial in each input.

Goodwillie, in \cite{goodcalc3}, suggested that his classification would be well-served by being revised using the structure and language of model categories and hence phrased in terms of Quillen equivalences.
For the homotopy functor calculus, Biedermann, Chorny  and R\"ondigs \cite{BCR07} and Biedermann and R\"ondigs \cite{BRgoodwillie} completed Goodwillie's recommendation.
For simplicial functors with fairly general target and domain,
they follow the same pattern as Goodwillie's paper \cite{goodcalc3}.
This classification involves several intermediate categories, similar to Figure \ref{fig:homogclass}.
In Figure \ref{fig:BR},  $\scal$ denote based simplicial sets, $\scal^f$ denotes finite based simplicial sets and
$\Fun(\Sigma_n \wr (\scal^f)^{\smashprod n}, C)_{\ml}$ denotes a model
structure of symmetric multi-linear functors with target $C$ being simplicial sets or spectra.
\begin{figure}[H]
\[
\xymatrix@C+1.5cm{
&
\Fun(\Sigma_n \wr (\scal^f)^{\smashprod n}, \scal)_{ml}
\ar@<3pt>[r]^{(-)/\Sigma_n \circ \Delta_n}
\ar@<-3pt>[d]_{\Sigma^\infty}
&
\Fun(\scal^f, \scal)_{\nhomog}
\ar@<3pt>[l]^{\cref_n}
\ar@<-3pt>[d]_{\Sigma^\infty}
\\
{\Sigma_n} \lca \Sp
\ar@<3pt>[r]^{L_{\textrm{eval}_{S^0}}}
&
\Fun(\Sigma_n \wr (\scal^f)^{\smashprod n}, \Sp)_{\ml}
\ar@<-3pt>[u]_{\Omega^\infty}
\ar@<3pt>[r]^{(-)/\Sigma_n \circ \Delta_n}
\ar@<3pt>[l]^{\textrm{eval}_{S^0}}
&
\textrm{Fun}(\scal^f, \Sp)_{\nhomog}
\ar@<3pt>[l]^{\cref_n}
\ar@<-3pt>[u]_{\Omega^\infty}
}
\]
\caption{Classification of Biedermann-R\"ondigs for $\ccal = \scal^f, \dcal=\scal$ \cite[(6.2)]{BRgoodwillie}}
\label{fig:BR}
\end{figure}
For the orthogonal calculus, the classification of $n$-homogeneous functors
by Weiss \cite[Section 7]{weiss95} was re-worked and promoted to a description in terms of Quillen equivalences of model categories in Barnes and Oman \cite{barnesoman13}. In the notation of this paper, their classification diagram (\cite[p.962]{barnesoman13}) is Figure \ref{fig:orthclass}.
Without going into detail, the left hand category is a model structure for $n$-homogeneous functors and the right hand category is a spectra with an action of $O(n)$.

\begin{figure}[H]
\begin{center}
$\xymatrix{
(n\text{-homog-Fun}(\jcal_0, \Top)) \ar@<5pt>[r]&\orthdevcat  \ar[l]  \ar@<5pt>[r]&O(n)\circlearrowleft \text{Sp} \ar[l]
}$
\end{center}
\caption{Weiss's classification}
\label{fig:orthclass}
\end{figure}

\subsection{Re-working the classification}
The middle category of Figure \ref{fig:orthclass} is not a kind of orthogonal version of symmetric multilinear functors.
Indeed, there appears to be no such analog in the orthogonal setting.
With that in mind, as well as our goal of a model-category comparison of the two calculi, we re-work the homogeneous classification for homotopy functors without using symmetric multilinear functors.
We instead use the homotopy functor analog of the middle category of Figure \ref{fig:orthclass}, which we denote $\devcat$.  This notation reflects our choice to use the category $\wcal \Top$
(continuous functors from finite based CW-complexes to based topological spaces)
as our model for homotopy functors from spaces to spaces, which we will say more about later.

In this paper, we construct the diagram of Quillen equivalences of Figure \ref{fig:fig}. The top line of this diagram provides an alternate classification of $n$-homogeneous functors and is analogous to Figure \ref{fig:orthclass}. We also compare our classification with the model category of symmetric multi-linear functors.
\begin{figure}[H]
\[
\xymatrix@C+1cm@R+1.2cm{
{\wcal \Top_{\nhomog}}
\ar@<3pt>[d]^-{\cref_n}
\ar@<-3pt>[r]_-{\diff_n}
&
\ar[]!<-1ex,0ex>;[dl]!<-1ex,0ex>_{L_{\obdiag}}
\devcat_{\stable}
\ar@<3pt>[r]^-{\wcal \smashprod_{\wcal_n} -}
\ar@<-3pt>[l]_-{\qquad (-)/\Sigma_n \circ \mapdiag^\ast}
&
\ar@<3pt>[l]^-{\phi_n^\ast}
\Sigma_n\circlearrowleft \wcal Sp
\\
\symfun( \wcal^n, \Top)_{\ml}
\ar@<4pt>[u]^-{(-)/\Sigma_n \circ \Delta_n}
\ar[]!<2ex,0ex>;[ur]!<2ex,0ex>_-{\obdiag^*}
}
\]
\caption{Diagram of Quillen equivalences}
\label{fig:fig}
\end{figure}

We show that the derivative construction (denoted $\diff_n$, see Definition \ref{def:cref}) in the setting of spaces over a point naturally takes values in $\devcat$ and that this construction is a Quillen equivalence. We furthermore construct a Quillen equivalence between $\devcat$ and spectra with a $\Sigma_n$-action (denoted $\Sigma_n\circlearrowleft \wcal Sp$).
Our new classification then resembles the orthogonal version of Barnes and Oman \cite{barnesoman13},
which involves one less adjunction and fewer categories than that of \cite{BRgoodwillie} and \cite{goodcalc3}.

The category $\devcat$ is a relatively standard construction
of equivariant spectra, similar to the constructions of equivariant
orthogonal spectra of Mandell and May \cite{mm02}. If we are prepared to
work with this category rather than spectra with a $\Sigma_n$-action
we have a one-stage classification of homogeneous functors in terms of spectra.
We also claim that our category $\devcat$ is no more complicated than the category of symmetric functors. See Section \ref{sec:intermediate}
 for a definition of $\devcat$ and Section \ref{sec:comp-symm-lin} for a comparison with symmetric multi-linear functors.

Another useful aspect of this work is that we choose to work with the category $\wcal\Top$ as our model of homotopy functors. Every object of this category is a homotopy functor, which removes the need for the homotopy functor model structure, a prominent feature of \cite{BCR07, BRgoodwillie}. We comment more on this in Section \ref{subsec:catwtop}.

\paragraph{In a sequel to this paper:}
Given a functor $F \in \wcal \Top$ we can consider
the functor of vector spaces $V \mapsto F (S^V)$ (where $S^V$ 
is the one-point compactification of $V$),
which we call the restriction of $F$.
We show that the
restriction of an $n$-homogeneous functor
(in the sense of Goodwillie) gives an $n$-homogeneous functor
(in the sense of Weiss).
Similarly, we show that
restriction sends $n$-excisive functors to $n$-polynomial functors.
These statements currently have the status of folk-results;
we will provide formal proofs in \cite{barneseldred15}.

Our primary aim in the sequel is to show that when $F$ is analytic
the restriction of the Goodwillie tower of $F$
and the Weiss tower associated to the
functor $V \mapsto F (S^V)$ agree.
From this, we obtain two applications. Firstly, we prove convergence of the
Weiss tower of the functor $V \mapsto BO(V)$ 
(as claimed in \cite[Page 13]{aroneweiss}).
Secondly, we lift the comparisons of the two forms of calculus to
a commutative diagram of model categories and Quillen pairs, see \cite[Section 5]{barneseldred15}.

With this aim in mind, working with a \textit{topologically enriched} category of homotopy functors rather than \textit{simplicially enriched} is necessary:
there is no good way to study orthogonal calculus using simplicial enrichments,
due to the continuity of the $O(n)$ actions.
Similarly, while \cite{BRgoodwillie} considers the case of homotopy functors between
categories other than simplicial sets or spectra, there is no
analogous generalisation for orthogonal calculus (since the domain is
the category $\jcal_0$ of real inner product spaces and linear isometries).
As our overall aim is a comparison between these two kinds of calculus,
we choose to work in the specific context of $\wcal \Top$ in this paper.

\subsection{Organisation}

In Section \ref{sec:modelgoodwillie} we remind the reader of some
important model category definitions and introduce $\wcal \Top$, the category of
functors that we will use to model homotopy functors.
We then follow the structure of \cite{BRgoodwillie} and establish
model structures on $\wcal \Top$ analogous to their work.
Specifically, in Section \ref{sec:calcmod} we define the cross effect model structure, the $n$-excisive model structure and the $n$-homogeneous model
structure.

With these basics completed, we can turn to the construction of the new category $\devcat$.
In Section \ref{sec:intermediate}, we start by giving the construction of the stable model structure on spectra with a $\Sigma_n$--action and then move on to constructing the stable model structure on $\devcat$. Section \ref{sec:compare} establishes the Quillen equivalence between $\devcat$ and spectra with a $\Sigma_n$--action. The Quillen equivalence between $n$-homogeneous functors and $\devcat$ induced by differentiation is established in Section \ref{sec:diffquillen}.  This is the primary result of this part of the paper. We finish by giving the Quillen equivalence between symmetric multilinear functors and $\devcat$ in Section \ref{sec:comp-symm-lin}.

\paragraph{Acknowledgements}
Part of this work was completed while the first author was supported by an EPSRC grant (EP/H026681/1), and the latter by the Humboldt Prize of Michael Weiss.  The authors would also like to thank Greg Arone,
Georg Biedermann, Oliver R\"ondigs and Michael Weiss for a number
of stimulating and helpful discussions.

%
%
\section{Model structures on spaces and functors}\label{sec:modelgoodwillie}
%

\subsection{Model category background}\label{subsec:modcats}

The conditions we use are essentially those which make arbitrary model categories most like spaces:
the ability to pushout or pullback weak equivalences (properness) and a good notion of cellular approximation (cofibrantly generated). We take our definitions and results from Hirschhorn \cite{hir03} and May and Ponto \cite{mp12}.

Similarly to Mandell et al.\ \cite{mmss01}, we use \emph{topological}, rather than simplicial model categories. When we say a \emph{category} is topological we mean that it is enriched in $\Top$ in the sense of Kelly \cite{kell05} (the category has spaces of morphisms and continuous composition).
Whereas a \emph{model category} is said to be topological if it satisfies the following definition, which is analogous to the concept of a simplicial model category.

\begin{definition} \label{def:topological}\cite[Definition 5.12]{mmss01}
 For maps $i: A \ra X$ and $p: E \ra B$ in a model category $\mathcal{M}$, let the map below
be the map of spaces induced by $\mathcal{M}(i, id)$ and  $\mathcal{M}(i, p)$ after passing to the pullback.
\[
\mathcal{M} (i^\ast, p_\ast): \mathcal{M}(X, E) \ra \mathcal{M} (A, E) \x_{\mathcal{M} (A,B)} \mathcal{M} (X,B)
\]

A model category $\mathscr{M}$ is \textbf{topological}, provided that $\mathcal{M} (i^\ast, p_\ast)$ is a Serre fibration of spaces if $i$ is a cofibration and $p$ is a fibration; it is a weak equivalence if, in addition, either $i$ or $p$ is a weak equivalence.
\end{definition}

\begin{definition}\cite[Definition 11.1.1] {hir03}
Let $\mathcal{M}$ be a model category, and let the following be a commutative square in $\mathcal{M}$:
\[
\xymatrix{
A \ar[r]^f \ar[d]_i & B \ar[d]^j\\
C \ar[r]_g & D\\
}.
\]
$\mathcal{M}$ is called \textbf{left proper} if, whenever $f$ is a weak equivalence, $i$ a cofibration, and the square is a pushout, then $g$ is also a weak equivalence.
$\mathcal{M}$ is called \textbf{right proper} if, whenever $g$ is a weak equivalence, $j$ a fibration, and the square is a pullback, then $f$ is also a weak equivalence.
$\mathcal{M}$ is called \textbf{proper} if it is both left and right proper.
\end{definition}
This concept can also be phrased as the set of (co)fibrations being closed under (co)base change.

\begin{definition}\cite[Definition 13.2.1]{hir03}
A \textbf{cofibrantly generated model category} is a model category $\mcal$ with sets of maps $I$ and $J$ such that $I$ and $J$ support the small object argument (see \cite[Definitions 15.1.1 and 15.1.7.]{mp12}) and
\begin{enumerate}
\item a map is a trivial fibration if and only if it has the right lifting property with respect to every element of $I$, and
\item a map is a fibration if and only if it has the right lifting property with respect to every element of $J$.
\end{enumerate}
\end{definition}

\subsection{Model structures on spaces}
There are three model structures that we use on $\Top$, the $q$-(``Quillen") model structure, %
the $h$-(``Hurewicz) model structure and the $m$-(``mixed") model structure.

\begin{theorem}\cite[Theorem 17.1.1, Corollary 17.1.2]{mp12}
The category $\Top$ of based spaces has a monoidal and proper
model structure, the \textbf{$h$-model structure}, where the
weak equivalences are the homotopy equivalences;
the fibrations are the Hurewicz fibrations and the
cofibrations are the $h$-cofibrations (those  maps with the homotopy extension property).
All spaces are both fibrant and cofibrant.
\end{theorem}

\begin{theorem}\cite[Theorem 17.2.2, Corollary 17.2.4]{mp12}
The category $\Top$ of based spaces has a cofibrantly
generated monoidal and proper model structure, the \textbf{$q$-model structure}, where
the weak equivalences are the weak homotopy equivalences;
the fibrations are the Serre fibrations (those maps that satisfy the
right lifting property with respect to $J_{\Top}$ as defined below).
The cofibrations are the $q$-cofibrations (defined by the left
lifting property).
All spaces are fibrant.
\end{theorem}
The $q$-model structure on spaces is cofibrantly generated.
The generating cofibrations ($I_{\Top}$) are the inclusions
$S^{n-1}_+ \ra D^n_+$, $n \geq 0$ and the
generating acyclic cofibrations ($J_{\Top}$)
are the maps $i_0 : D^n_+\ra (D^n \x I)_+$, $n \geq 0$.

\begin{theorem} \cite[Theorem 17.4.2, Corollary 17.4.3]{mp12}
The category $\Top$ of based spaces has a monoidal and proper model structure, the \textbf{$m$-model structure}, where the
weak equivalences are the weak homotopy equivalences;
the fibrations are the Hurewicz fibrations
and the cofibrations defined by the left
lifting property with respect to Hurewicz fibrations which are also $q$-equivalences.
\end{theorem}

Note that every $m$-cofibration is a $h$-cofibration and
the $h$-cofibrations are closed inclusions of spaces, see
Mandell et al. \cite[Page 457]{mmss01}.

\subsection{The category \texorpdfstring{$\wcal \Top$}{WTop} of topological functors}\label{subsec:catwtop}

Goodwillie calculus studies equivalence-preserving functors from
the category of based spaces to itself. %
In this section we introduce $\wcal \Top$ and show how it is
a good model for these.

Let $\wcal$ be the category of based spaces homeomorphic to finite CW complexes. We note immediately that $\wcal$ is $\Top$-enriched, but not $\wcal$-enriched. We define $\wcal \Top$ to be the category of \textbf{$\wcal$-spaces}: continuous functors from
$\wcal$ to $\Top$ (for full details see Mandell et al.\ \cite{mmss01}).
In particular, an $X \in \wcal \Top$ consists of the following information:
a collection of based spaces $X(A)$
for each $A \in \wcal$ and a collection of maps
of based spaces
\[
X_{A,B} \co \wcal(A,B) \longrightarrow \Top(X(A), X(B))
\]
for each pair $A$, $B$ in $\wcal$.
These maps must be compatible with composition and also
associative and unital.
The map $X_{A,B}$ induces a structure map:
\[
X(A) \smashprod \wcal(A,B) \longrightarrow X(B)
\]
The category $\wcal \Top$ is complete and cocomplete with limits and colimits taken objectwise.
This category is tensored and cotensored over based spaces.
For a functor $X$ in $\wcal \Top$ and a based space $A$, the tensor $X \smsh A$ is
the objectwise smash product. The cotensor $\Top(A,X)$ is the objectwise function space.
The category $\wcal \Top$ is also enriched over based spaces, with the space
of natural transformations from $X$ to $Y$ given by the enriched end (for more on (co)ends, see Kelly \cite[Section 3.10]{kell05})
\[
\nat(X,Y) = \int_{A \in \wcal} \Top(X(A), Y(A))
\]
The category $\wcal \Top$ is a closed symmetric monoidal category
by Mandell et al. \cite[Theorem 1.7]{mmss01}.
The smash product and internal function object are defined as
follows, where $X$ and $Y$ are objects of $\wcal \Top$ and
$A \in \wcal$.
\[\begin{array}{rcl}
(X \smsh Y)(A) & = & \int^{B,C \in \wcal}
X(B) \smsh Y(C) \smsh \wcal(B \smsh C, A) \\
\Hom(X,Y)(A) & = & \int_{B \in \wcal} \Top(X(B),Y(A \smsh B))
\end{array}
\]
There is another important natural construction that we will use.
Let $X$ be an object of $\wcal \Top$,
then the \textbf{assembly map} of $X$ is
\[
a_{A,B} \co X(A) \smashprod B \to X(A \smashprod B).
\]
It may be defined as the following composition,
where the final map is the structure map of $X$.
\[
\xymatrix{
X(A) \smashprod B
\cong
X(A) \smashprod \wcal(S^0, B)
\ar[rr]^-{\id \smashprod (A \smashprod -)}&&
X(A) \smashprod \wcal(A, A \smashprod B)
\ar[r]&
X(A \smashprod B)
}
\]
The existence of the assembly map tells us that
$X$ takes homotopic maps to homotopic maps (compose
the assembly map with $X$ applied to the homotopy between the maps). Since $\wcal$ consists of CW--complexes, it follows that
$X$ preserves weak homotopy equivalences, that is,
$X$ is a homotopy functor.

We record here an important observation about objects of
$\wcal \Top$.
Since $\id_\ast$ is the basepoint of $\wcal (\ast, \ast)$,
the map $\id_{X(\ast)} = X(\id_\ast)$ is the base point of
$\Top(X(\ast), X(\ast))$. Hence $X(\ast) = \ast$
for any $X \in \wcal$. We therefore say that every functor of
$\wcal \Top$ is \textbf{reduced}.

The category $\wcal$ has a small skeleton $\skel \wcal$, which fixes set-theoretic problems with the totality of natural transformations between functors from $\Top$ to $\Top$.
In particular, it ensures that all small limits exist in $\wcal \Top$.
Biedermann and R\"ondigs \cite{BRgoodwillie}
work (in particular) with the simplicial analogue of $\wcal \Top$
and  considers Goodwillie calculus in terms of
simplicial functors from the category of finite simplicial sets $\scal^f$ to the category of all simplicial sets $\scal$.
A nice discussion of the set-theoretic problem can be
found in Biedermann-Chorny-R\"{o}ndigs \cite[Section 2]{BCR07}.

We now want to equip the category $\wcal \Top$ with a model structure,
the following result is  \cite[Theorem 6.5]{mmss01}.

\begin{lemma}\label{lem:projmod}
The \textbf{projective model structure} on the category $\wcal \Top$ has fibrations and
weak equivalences which are defined objectwise in the $q$-model structure of spaces.
The cofibrations are determined by the left lifting property.
In particular they are objectwise $m$-cofibrations of spaces.
This model structure is proper,
cofibrantly generated and  topological.
The generating sets are given below, where
$\skel \wcal$ denotes a skeleton of $\wcal$.
\[
\begin{array}{rcl}
I_{\wcal \Top} & = & \{
\wcal(X,-) \smsh i \ | \ i \in I_{\Top}, X \in \skel \wcal  \} \\
J_{\wcal \Top} & = & \{
\wcal(X,-) \smsh j \ | \ j \in J_{\Top}, X \in \skel \wcal  \}
\end{array}
\]
\end{lemma}

Recall from Goodwillie \cite[Definition 5.10]{goodcalc3}
that a homotopy functor from $\Top$ to $\Top$ is said to be
\textbf{finitary} if it commutes with filtered
homotopy colimits. %
Such functors are determined by  their restriction to $\wcal$.
Since any space $A$ is naturally weakly equivalent to a homotopy colimit
of finite CW-complexes $\hocolim_n A_n$, we can extend a homotopy functor $X \in \wcal \Top$
by the formula $X(A) = \hocolim_n X(A_n)$ to obtain a
finitary homotopy functor from $\Top$ to itself.

To relate $\wcal \Top$ to the work of Biedermann and R\"ondigs,
consider the category of simplicial functors from the category
of finite based simplicial sets to the category of based simplicial sets,
$\Fun(\scal^{f}, \scal)$. This category can be equipped the
homotopy functor model structures of \cite[Section 4]{BRgoodwillie}.
It is then an exercise left to the enthusiast to show that
$\wcal \Top$ with its projective model structure is Quillen equivalent to
$\Fun(\scal^{f}, \scal)$ with the homotopy functor model structure. The result is
a consequence of the simplicial approximation theorem, which implies that a finite CW
complex is homotopy equivalent to the realisation of a finite simplicial complex.

\section{Model structures for Goodwillie calculus}\label{sec:calcmod}

In this section, we explain how to construct model
categories of $n$-excisive functors and $n$-homogeneous functors.
Only brief details are given, as the method is similar to that of \cite{BRgoodwillie}
and Barnes and Oman \cite{barnesoman13}.  Many of the following constructions and definitions may be found originally in \cite{goodcalc3} .

%
\subsection{The cross effect model structure}\label{subsec:crefmod}
%

We need the cross effect and the functor $\diff_n$ (defined below)
to be right Quillen functors for the
classification of the $n$-homogeneous functors.
That is, if $f \co F \to G$ is a fibration of $\wcal \Top$, we need
$\diff_n(f):\diff_n F \to \diff_n G$
(and the same for the cross effect) to be an
objectwise fibration of $\wcal \Top$. This does not hold for the
projective model structure, as explained in the introduction to
\cite[Section 3.3]{BRgoodwillie}.

Similar to Biedermann and R\"ondigs (albeit topologically rather than simplicially), we introduce another model structure
on $\wcal \Top$ that is Quillen equivalent to the
projective model structure. This alternative model
structure will be called the \textbf{cross effect model structure} (see Theorem \ref{thm:crefmodel}),
It has the same weak equivalences as the projective model structure.

\begin{definition} \label{def:cref}
For $F \in \wcal \Top$ and an $n$-tuple of spaces in $\wcal$, $(X_1, \ldots, X_n)$, the \textbf{$n^{th}$--cross effect of $F$
at $(X_1, \ldots, X_n)$} is the space
\[
\cref_n(F)(X_1, \ldots, X_n) = \nat(\bigsmashprod{l=1}{n} \wcal(X_l,-), F)
\]
Pre-composing $\cref_n(F)$ with the diagonal map
$\wcal(X,Y) \to \bigsmashprod{i=1}{n} \wcal(X,Y)$
yields an object of $\wcal \Top$
which we call $\diff_n(F)$, which in keeping with language of orthogonal calculus, is the
$n^{th}$ \textbf{(unstable) derivative}.
That is
\[
\diff_n(F)(X) = \nat(\bigsmashprod{l=1}{n} \wcal(X,-), F).
\]
\end{definition}

In Section \ref{sec:diffquillen} we elaborate on how the spaces $(\diff_n F)(X)$
define a spectrum. As it is defined in terms of the cross-effect we only work in the based setting, so it is the derivative over the point.

\begin{remark}\label{rmk:hostuff}
We caution the reader that there is a difference between what Goodwillie \cite{goodcalc3} calls the $n^{th}$ cross-effect and the above notation. As is now standard, Goodwillie's version
is called the \emph{homotopy} cross-effect.
\end{remark}


To make the cross effect into a right Quillen functor we need to have more cofibrations than in the projective model structure on $\wcal \Top$.
The extra maps we need are defined below in Definition \ref{def:maps}. We first need the following formalism for cubical diagrams.

\begin{definition}
Let $\underline{n}$ denote the set $\{ 1, \dots, n \}$ and let
$\pcal(\underline{n})$ denote the powerset of $\underline{n}$.
We define $\pcal_0(\underline{n})$ as the set of
non--empty subsets of $\underline{n}$.
\end{definition}

\begin{definition}\label{def:maps}
Consider the following collection of maps, where $\phi_{\underline{X},n}$ is defined via the projections which send those factors in $S$ to the basepoint.
\[
\begin{array}{c}
\Phi_n = \{  \phi_{\underline{X},n}
\co
\underset{S \in \pcal_0(\underline{n})}{\colim}
\wcal({\bigvee_{l \in \underline{n} - S} X_l}, -)
\longrightarrow
\wcal({\bigvee_{l =1}^n X_l},-)
\ | \ \underline{X} = (X_1, \dots , X_n), X_l \in \skel \wcal  \} \\
\end{array}
\]
We then also define $\Phi_\infty = \cup_{n \geqslant 1} \Phi_n$.
\end{definition}
The cofibre of $\nat(-,F) (\phi_{\underline{X},n})$ is the cross effect
of $F$ at $\underline{X}$, $\cref_n(F)(X_1, \ldots, X_n)$;
see \cite[Lemma 3.14]{BRgoodwillie}.

\begin{definition}
Given $f \co A \to B$ a map of based of spaces and $g \co X \to Y$ in $\wcal$,
the \textbf{pushout product} of $f$ and $g$, $f \Box g$, is given by
\[
f \Box g \co
B \smsh X \bigvee_{A \smsh X} A \smsh Y \to B \smsh Y.
\]
\end{definition}

\begin{theorem}\label{thm:crefmodel}
There is a cofibrantly generated model structure on $\wcal \Top$, \textbf{the cross effect model structure},
whose weak equivalences are the objectwise weak homotopy equivalences
and whose generating sets are given by
\[
I_{\wcal cr}  = \Phi_\infty \Box I_{\Top} \quad J_{\wcal cr} = \Phi_\infty \Box J_{\Top}
\]
We call the cofibrations of this model structure \textbf{cross effect cofibrations}
and call the fibrations the \textbf{cross effect fibrations}.
We write $\wcal \Top_{\cross}$ for this model category and $\fibrep_{\cross}$ for its fibrant
replacement functor.
\end{theorem}

\begin{proof}
Similar to the arguments of \cite[Section 3.3]{BRgoodwillie}.
Note the following two facts:
\begin{enumerate}
\item $\phi_{\underline{X},n}$ is an
objectwise $m$-cofibration (and hence a $h$-cofibration)
of based spaces,
\item the domains of the generating sets are small with respect to the objectwise $h$-cofibrations by Hovey \cite[Proposition 2.4.2]{hov99}
and Hirschhorn \cite[Proposition 10.4.8]{hir03}.
\end{enumerate}
\end{proof}

\begin{corollary}\label{cor:crossproperty}
The cross effect model structure on $\wcal \Top$ is proper
and the cofibrant objects
are small with respect to the class of objectwise $h$--cofibrations.
\end{corollary}
\begin{proof}
Every cross effect cofibration is an objectwise $h$--cofibration.
Similarly every cross effect fibration is an objectwise $q$-fibration.
Since the weak equivalences, limits and colimits are all defined objectwise,
the result follows from standard properties of $\Top$.
The smallness follows from \cite[Section 10.4]{hir03} and the second
point of the proof of Theorem \ref{thm:crefmodel}.
\end{proof}

\begin{corollary}\label{cor:wncofibrations}
For $k \co A \to B$ a cofibration of based spaces and $(X_1, \dots, X_n)$ an $n$-tuple of
objects of $\wcal$,  the map
\[
\bigsmashprod{l=1}{n} \wcal(X_l,-) \smashprod k \co \bigsmashprod{l=1}{n} \wcal(X_l,-) \smashprod A \to \bigsmashprod{l=1}{n} \wcal(X_l,-) \smashprod B
\]
is a cross effect cofibration.
\end{corollary}
\begin{proof}
The map $\alpha \co \ast \to \bigsmashprod{l=1}{n} \wcal(X_l,-)$
is a cross effect cofibration, where $\ast$ denotes here the one point space.
It follows that $\alpha \Box k$ is a cross effect cofibration.
\end{proof}

The proof of the following is effectively Biedermann and R\"ondigs \cite[Lemma 3.24]{BRgoodwillie}.

\begin{lemma}\label{lem:crossfibrant}
If $F$ is a cross effect fibrant object of $\wcal \Top$ then
the $n^{th}$ homotopy cross effect of $F$ is given by the
strict $n^{th}$ cross effect.
\end{lemma}

%
\subsection{The \texorpdfstring{$n$}{n}-excisive model structure}\label{subsec:nexcmod}
%

As in Barnes and Oman \cite[Section 6]{barnesoman13},
we perform a left Bousfield localisation of the
cross effect model structure on $\wcal \Top$ to obtain the
$n$-excisive model structure.
The class of fibrant objects
of this model structure will be the class of $n$-excisive
objects of $\wcal \Top$; see Definition \ref{def:nexc} and Theorem \ref{thm:nexc-model}.
The cofibrations will remain unchanged
and the weak equivalences will be the \textbf{$P_n$-equivalences},
those maps $f \co F \to G$ such that $P_n f$ (see Definition \ref{defn:Tn}) is an
objectwise weak homotopy equivalence.

\begin{definition}
An \textbf{$n$--cube} in $\wcal$ (or $\Top$) is a functor
$\mathcal{X}$ from $\pcal(\underline{n})$ to $\wcal$ (resp. $\Top$).
An $n$--cube is said to be \textbf{strongly cocartesian} if
all of its two-dimensional faces are homotopy pushout
squares.
An $n$--cube is said to be \textbf{cartesian} if
the map
\[
\mathcal{X}(\emptyset) \longrightarrow \holim_{S \in \pcal_0(\underline{n})} \mathcal{X}(S)
\]
induced by the maps $\mathcal{X}(\emptyset) \to \mathcal{X}(S)$ is a
weak homotopy equivalence.
\end{definition}

\begin{definition}\label{def:nexc}
An object $F \in \wcal \Top$ is said to be \textbf{$n$-excisive}
if it sends strongly cocartesian $(n+1)$-cubes in $\wcal$ to cartesian $(n+1)$-cubes in $\Top$.
\end{definition}

We now give the construction of the homotopy-universal approximation to $F$ by an $n$-excisive functor, denoted $P_n F$.
Note that we use $X \join Y$ to denote the topological join of $X$ and $Y$.

\begin{definition}\label{defn:Tn}
We first define a functor $T_n \co \wcal \Top \to \wcal \Top$
and a natural transformation $t_{n} \co \id \to T_n$.
Let $F \in \wcal \Top$ then $T_n F$ is given below.
\[
(T_n F)(X)  =
\Nat( \underset{S \in \pcal_0(\underline{n+1})}{\hocolim} \wcal(S \join X,-), F)
=
\underset{S \in \pcal_0(\underline{n+1})}{\holim}
F(S \join X)
\]
The inclusion of the empty set as the initial object of $\pcal_0(\underline{n+1})$ and that $\emptyset \join X \cong X$ gives a natural transformation $t_{n,F}$ from $F(-) \cong F(\emptyset\join -) $ to the homotopy limit $T_nF$.

Furthermore, we define
\[
P_n F  := \hocolim
\big(
F
\overset{t_{n,F}}{\longrightarrow}
T_n F
\overset{t_{n,T_n F}}{\longrightarrow}
T_n^2 F
\overset{t_{n,T_n^2 F}}{\longrightarrow}
T_n^3 F
\longrightarrow \dots
\big)
\]
\end{definition}


For more details on homotopy limits in functor categories see Heller \cite{heller82}.
In particular $T_n F$ and $P_n F$ are continuous functors from $\wcal$
to based topological spaces. One could also apply model category techniques and take a strict limit of a suitably fibrant replacement of the diagram.

The proof of Goodwillie \cite[Theorem 1.8]{goodcalc3} implies the following
result.

\begin{lemma}
An object $F \in \wcal \Top$ is $n$-excisive if and only if the map $t_{n,F} \co F \to T_n F$
is an objectwise weak homotopy equivalence.
\end{lemma}

In particular, $t_{n,F} \co F \to T_n F$ is a $P_n$-equivalence for any $F$.
To make a new model structure where
the $P_n$--equivalences are weak equivalences,
it is necessary and sufficient to turn the class of maps
$t_{n,F}$ into weak equivalences.
Consider the following \emph{set} of maps
\[
S_{n}  =  \{
s_{n,X} \co \underset{S \in \pcal_0(\underline{n+1})}{\hocolim}
\wcal({S \join X}, -)
\longrightarrow
\wcal( X,-) \ | \ X \in \skel{\wcal}
\}.
\]
By the Yoneda lemma, $\nat(-,F)(s_{n,X}) \simeq t_{n,F}(X)$.
Hence, a model structure on $\wcal \Top$ will have $S_n$
contained in the
weak equivalences if and only if  the $P_n$-equivalences
are weak equivalences. We proceed to alter $\wcal \Top$
so that the maps in $S_n$ are weak equivalences.

We replace the set of maps $S_n$ by a set of
objectwise $h$-cofibrations, $K_n$.
For $s_{n,X} \in S_n$ let $k_{n,X}$ be the map from the domain of $s_{n,X}$ into the mapping cylinder $Ms_{n,X}$. Similarly, let
$r_{n,X}\colon Ms_{n,X} \to \wcal( X,-)$
be the retraction.
Define
\[
K_n  =  \{
k_{n,X} \co \underset{S \in \pcal_0(\underline{n})}{\hocolim}
\wcal({S \join X}, -)
\longrightarrow
Ms_{n,X} \ | \ X \in \skel \wcal
\}.
\]

\begin{theorem}\label{thm:nexc-model}
There is a cofibrantly generated model structure on $\wcal \Top$
whose weak equivalences are the $P_n$-equivalences
and whose generating sets are given by
\[
I_{\nexs}  = \Phi_\infty \Box I_{\Top} \quad J_{\nexs} =
\left( \Phi_\infty \Box J_{\Top} \right)
\cup \left( K_n \Box I_{\Top} \right)
\]
The cofibrations are the \textbf{cross effect cofibrations}
and the fibrations are called \textbf{$n$--excisive fibrations}.
In particular, every $n$--excisive fibration is a cross effect fibration.
The fibrant objects are the cross effect fibrant
$n$--excisive functors.
We write $\wcal \Top_{\nexs}$ for this model category,
which we call the \textbf{$n$--excisive model structure}.
\end{theorem}
\begin{proof}
Much of the work is similar to Biedermann and R\"ondigs
\cite[Theorem 5.8 and Lemma 5.9]{BRgoodwillie}.
The lifting properties and classification of the weak equivalences are consequences of the following statement:
a map $f$ has the right lifting property with respect to
$J_{\nexs}$ if and only if $f$ is a cross effect fibration and
either (and hence both) of the squares below is
a homotopy pullback for all $X \in \wcal$.
\[
\xymatrix{
F(X)
\ar[r] \ar[d] &
(T_n F )(X)
\ar[d] \\
G(X) \ar[r] &
(T_n G )(X)
}
\qquad \qquad
\xymatrix{
F(X)
\ar[r] \ar[d] &
(P_n F )(X)
\ar[d] \\
G(X) \ar[r] &
(P_n G )(X)
}
\]
The small object argument holds in this setting  by Hirschhorn \cite[Theorem 18.5.2]{hir03}
and Corollary \ref{cor:crossproperty}.
\end{proof}

\begin{proposition}\label{prop:proper}
The $n$--excisive model structure on $\wcal \Top$ is
proper.
\end{proposition}

\begin{proof}
The functor $P_n$ satisfies the assumptions of Bousfield
\cite[Theorem 9.3]{bous01}
(as verified in \cite[Theorem 5.8]{BRgoodwillie}).
Hence, there is a proper model structure
on $\wcal \Top$ with weak equivalences the $P_n$--equivalences
and cofibrations the cross effect cofibrations -- which is
precisely our $n$-excisive model
structure, so it is proper.
\end{proof}

Note that every $n$--excisive functor in $\wcal \Top$ is
objectwise weakly equivalent to a cross effect fibrant
$n$-excisive functor.

\begin{lemma}\label{lem:fibrepnexs}
Fibrant replacement in $\wcal \Top_{\nexs}$
is given by first applying the functor $P_n$ and
then applying $\fibrep_{\cross}$, the fibrant replacement functor of
$\wcal \Top_{\cross}$.
\end{lemma}
\begin{proof}
For $F \in \wcal \Top$, $P_n F$ is $n$--excisive.
Applying $\fibrep_{\cross}$ we obtain an objectwise weakly equivalent
object $\fibrep_{\cross} P_n F$. This object is also $n$--excisive and is
cross effect fibrant. Hence it is fibrant in
$\wcal \Top_{\nexs}$. Thus we can set
$\fibrep_{\nexs}= \fibrep_{\cross} P_n$.
\end{proof}

%
\subsection{The \texorpdfstring{$n$}{n}-homogeneous model structure}\label{subsec:nhomog}
%

Our next class of functors to study are those which are `purely' $n$-excisive, that is, those $F$ such that $\P_n F \simeq F$ but $\P_{n-1} F \simeq *$.

\begin{definition}
An object $F \in \wcal \Top$ is said to be \textbf{$n$-homogeneous} if it is $n$-excisive and $P_{n-1} F(X)$ is weakly equivalent to a point for each $X \in \wcal$.

For $F \in \wcal \Top$, define $D_n F \in \wcal \Top$
as the homotopy fibre of $P_n F \to P_{n-1} F$.
Since $P_n$ and $P_{n-1}$ commute with finite homotopy limits
the functor $D_n F$ takes values in $n$-homogeneous functors
and hence is called the
\textbf{$n$-homogeneous approximation} to $F$.
\end{definition}

Similarly to Barnes and Oman \cite[Section 6]{barnesoman13} and
Biedermann and R\"ondigs \cite[Section 6]{BRgoodwillie}
we perform a right Bousfield localisation of $\wcal \Top_{\nexs}$
--this adds weak equivalences whilst preserving the class of fibrations.
The aim is to obtain a new model structure $\wcal \Top_{\nhomog}$ where the
weak equivalences are the $D_n$-equivalences and the cofibrant-fibrant objects are precisely the $n$--homogeneous objects which are fibrant
and cofibrant in the cross effect model structure.
Thus every $n$--homogeneous object of $\wcal \Top$ will be objectwise
weakly equivalent to a cofibrant-fibrant object of this new model structure.
This will give us a `short exact sequence' of model structures as below,
where the composite derived functor
$\wcal \Top_{\nhomog} \to \wcal \Top_{(n-1)\textrm{--nexs}}$
sends every object to the trivial object.
\begin{figure}[H]
\[
\xymatrix{
\wcal \Top_{\nhomog}
\ar@<+1ex>[r]
&
\wcal \Top_{\nexs}
\ar@<+1ex>[r]
\ar@<+1ex>[l]
&
\wcal \Top_{(n-1)\textrm{--nexs}}
\ar@<+1ex>[l]
}
\]
\caption{Sequence of localisations}
\label{fig:ses}
\end{figure}
The required $n$-homogeneous model structure %
will have weak equivalences those maps
$f \in \wcal \Top$ such that
\[
\hodiff_n f = \diff_n \fibrep_{\nexs} f = \diff_n \fibrep_{\cross} P_n f
\]
is an objectwise weak homotopy equivalence.
The name $\hodiff_n$ refers to the fact it is defined in terms of the
homotopy cross effect (see Remark \ref{rmk:hostuff} and
Lemma \ref{lem:crossfibrant}).
In Section \ref{sec:diffquillen} we shall turn the construction
$\diff_n$ into a Quillen functor (indeed, into a Quillen equivalence)
and $\hodiff_n$ will be its derived functor.

A pointed model category is called \textbf{stable}
if the suspension functor is an equivalence
on the homotopy category; this definition agrees with
Schwede and Shipley \cite[Definition 2.1.1]{ss03stabmodcat}.

\begin{theorem}
There is a model structure
on $\wcal \Top$ with fibrations
the $n$-excisive fibrations and weak equivalences
the $\hodiff_n$--equivalences.
We call this %
the \textbf{$n$--homogeneous model structure}
and denote it by $\wcal \Top_{\nhomog}$.
The model structure is
cofibrantly generated, proper and stable.
\end{theorem}
\begin{proof}
By Christensen and Isaksen \cite[Theorem 2.6]{CI04}
the right Bousfield localisation of the model category $\wcal \Top_{\nexs}$
at the set
\[
M_n = \{ \bigsmashprod{l=1}{n} \wcal(X,-) \ | \ X \in \skel \wcal \}
\]
exists and is right proper.
We have used the fact that cofibrantly generated
model categories (such as $\wcal \Top_{\nexs}$) always satisfy
\cite[Hypothesis 2.4]{CI04}.
Note that this set is substantially smaller than that of
\cite[Definition 6.2]{BRgoodwillie}, where the $X$ terms
depend on $l$.

The model category $\wcal \Top_{\nexs}$ is topological
(see Definition \ref{def:topological}).
Hence the weak equivalences of $\wcal \Top_{\nhomog}$
are given by those maps $f \co F \to G$
which induce weak homotopy equivalences of spaces
as below
for all $X \in \wcal$.
\[
\Nat (\bigsmashprod{l=1}{n} \wcal(X,-) , \fibrep_{\nexs} F )
\overset{\simeq}{\longrightarrow}
\Nat (\bigsmashprod{l=1}{n} \wcal(X,-) , \fibrep_{\nexs} G )
\]

By Lemma \ref{lem:fibrepnexs} there is a weak homotopy equivalence of spaces
\[
\Nat (\bigsmashprod{l=1}{n} \wcal(X,-) , \fibrep_{\nexs} F )
\simeq
\Nat (\bigsmashprod{l=1}{n} \wcal(X,-) , \fibrep_{\cross} P_n F )
=
\diff_n \fibrep_{\cross} P_n F = \hodiff_n F.
\]
It follows that the weak equivalences
of $\wcal \Top_{\nhomog}$ are as claimed.
In particular, every object of this model category is weakly equivalent to an
$n$-homogeneous functor.

The proof of Biedermann and R\"ondigs \cite[Theorem 6.11]{BRgoodwillie}
adapted to our setting shows that $\wcal \Top_{\nhomog}$ is stable.
Given that it is stable, the work of Barnes and Roitzheim \cite[Proposition 5.8]{barnesroitzheimstable}
tells us that the category is left proper.
A small variation on \cite[Theorem 5.9]{barnesroitzheimstable}
yields that the generating cofibrations are given by the union of the generating
acyclic cofibrations for $\wcal \Top_{\nexs}$ along with the set
of morphisms
\[
\{
S^k_+ \smashprod \bigsmashprod{l=1}{n} \wcal(X,-) \longrightarrow
D^k_+ \smashprod \bigsmashprod{l=1}{n} \wcal(X,-) \ | \ k \geqslant 0, \ X \in \skel \wcal
\}.
\]
\end{proof}

Our next task is to show that Figure \ref{fig:ses}
does indeed behave like a `short exact sequence'
of model categories. That is, we want to show that
an object $F$ of the $n$--excisive model structure has
$P_{n-1} F \simeq \ast$ if and only if it is in the image of
the derived functor from
$\wcal \Top_{\nhomog}$ to $\wcal \Top_{\nexs}$.

\begin{lemma}\label{lem:nhomogequivDn}
A map is a $\hodiff_n$--equivalence if and only if it is a
$D_n$--equivalence.
\end{lemma}

\begin{proof}
Let $f$ be a $D_n$-equivalence, so $D_n f$ is an objectwise weak homotopy equivalence.
Since $\hodiff_n f = \diff_n \fibrep_{\cross} P_n f$, it is weakly equivalent to
$\diff_n \fibrep_{\cross} D_n f$, the first half of the result follows.

For the converse, we use a method similar to Biedermann and R\"ondigs \cite[Lemma 6.19]{BRgoodwillie}.
Take some $\hodiff_n$--equivalence $f$.
We can extend this to a map $\Sigma^\infty f$ between
functors which take values in sequential spectra.
Applying $\hodiff_n$ levelwise to $\Sigma^\infty f$
gives an objectwise weak equivalence of spectra.

By Goodwillie \cite[Proposition 5.8]{goodcalc3} it follows that
$\hocref_n \Sigma^\infty f$ is an objectwise
weak equivalence of spectra.
The result \cite[Proposition 3.4]{goodcalc3} (see also \cite[Corollary 6.9]{BRgoodwillie}) implies that
$D_n \Sigma^\infty f$ is also an objectwise weak equivalence.
Hence so is the zeroth level of $D_n \Sigma^\infty f$,
$\textrm{Ev}_0 D_n \Sigma^\infty f$.
The functor $\textrm{Ev}_0$ commutes with $D_n$
(up to objectwise weak equivalence) and
$\textrm{Ev}_0 \Sigma^\infty \simeq \id$ since
we are in a stable model structure.
Thus $D_n f$ is an objectwise weak equivalence.
\end{proof}

We state the following without proof as it follows from \cite[Lemma 6.24]{BRgoodwillie}.

\begin{proposition}
An object of $\wcal \Top_{\nhomog}$ is cofibrant and fibrant if and only if
it is $n$--homogeneous and fibrant
and cofibrant in the cross effect model structure.
The cofibrations of $\wcal \Top_{\nhomog}$ are the cross effect cofibrations
that are $P_{n-1}$-equivalences.
\end{proposition}

Thus we now see that the cofibrant-fibrant objects of $\wcal \Top_{\nhomog}$
are exactly those functors of $\wcal \Top_{\nexs}$ that are trivial in
$\wcal \Top_{(n-1)\textrm{--nexs}}$. Thus Figure \ref{fig:ses}
is a `short exact sequence' of model categories.

%
\section{Capturing the derivative over a point}\label{sec:intermediate}
%

We begin this section by giving a stable model structure for the category of spectra with a $\Sigma_n$-action (as these classify the $n$-homogeneous functors).
It plays the role analogous to the intermediate category $O(n)\mathcal{E}_n$ of Barnes and Oman, see \cite[Section 7]{barnesoman13}.
This category has been designed to receive Goodwillie's derivative and
we shall show in Section \ref{sec:diffquillen} that the derivative is part of a Quillen equivalence.

After defining the category $\devcat$, we establish the projective model structure in Theorem \ref{thm:devcatprojmod}, then left Bousfield localise to get the stable structure. This makes use of the definition of $n \pi_*$-isomorphisms (analogous to  \cite[Definition 7.7]{barnesoman13}).

%
\subsection{A model category for spectra with a \texorpdfstring{$\Sigma_n$}{Sigma\textunderscore n}-action}\label{subsec:sigmanmodel}
%

One can model spectra by putting a stable model structure on $\wcal\Top$ as in
Mandell et al.\ \cite{mmss01}. This model category is Quillen equivalent to the other models of the
stable homotopy category. We perform the same operation but $\Sigma_n$-equivariantly.

The next result follows immediately from applying the
transfer argument, Hirschhorn \cite[Theorem 11.3.2]{hir03},
to the free functor  $(\Sigma_n)_+ \smsh -: \Top \ra \Sigma_n \circlearrowleft \Top$
where $\Top$ is equipped with the $q$-model structure.
See also Mandell and May \cite[Section II.1]{mm02}.

\begin{lemma}\label{lem:gspaces}
The category $\Sigma_n \lca \Top$ of based spaces with an action of $\Sigma_n$
has a cofibrantly generated monoidal and proper model structure.
The weak equivalences are those which are
weak homotopy equivalences after forgetting the $\Sigma_n$-action.
Similarly, the fibrations are those maps whose underlying map
in $\Top$ is a Serre fibration. The cofibrant objects are free.
The monoidal product is given equipping the smash product of
two $\Sigma_n$-spaces with the diagonal action.
The internal function object is given by equipping the space of
non-equivariant maps with the conjugation action:
if $f \in \Top(X,Y)$, $\sigma \cdot f = \sigma_Y \circ f \circ \sigma_X^{-1}$.
\end{lemma}

Combining the projective model structure on $\wcal \Top$ (Lemma \ref{lem:projmod}) with Lemma \ref{lem:gspaces},
we obtain the following model structure
on  $\Sigma_n\circlearrowleft \wcal \Top$, the category of
$\Sigma_n$--objects in $\wcal \Top$ and $\Sigma_n$--equivariant morphisms.

\begin{lemma}\label{lem:projmod-2}
The \textbf{projective model structure} on the category $\Sigma_n\circlearrowleft \wcal \Top$ has as generating sets
\[
\begin{array}{rcl}
I_{\Sigma_n \lca \Top}  & = & \{
\wcal(X,-) \smsh (\Sigma_n)_+ \smsh i \ | \ i \in I_{\Top}, X \in \skel \wcal  \} \\
J_{\Sigma_n \lca \Top} & = & \{
\wcal(X,-) \smsh (\Sigma_n)_+ \smsh j \ | \ j \in J_{\Top}, X \in \skel \wcal  \}.
\end{array}
\]

A fibration (resp. weak equivalence) in this model structure is a $\Sigma_n$-equivariant map $f$
such that each $f(X)$ is a $q$-fibration (resp. weak homotopy equivalence) of the underlying non-equivariant spaces.
If $F\in \Sigma_n\circlearrowleft \wcal \Top$ is cofibrant, then each $F(X)$ is a free $\Sigma_n$-space. This model structure is proper, cofibrantly generated and  topological.
\end{lemma}

We now modify the projective model structure to obtain the stable model structure. We first
relate $\Sigma_n\circlearrowleft \wcal \Top$ to sequential spectra, which
allows us to define the weak equivalences of the stable model structure.

\begin{definition}\label{def:underspec}
Let $F \in \Sigma_n\circlearrowleft \wcal \Top$ and $A \in \wcal$.
We define a spectrum $F[A]$  via
\[
F[A]_k : = F(A \smashprod S^k),
\]
where we have forgotten the $\Sigma_n$-action.
The assembly maps provide the structure maps of $F[A]$ as well as maps $F[A] \smashprod B \to F[A \smashprod B]$.
We call $F[S^0]$ the \textbf{underlying spectrum of $F$}.
\end{definition}

\begin{definition}\label{def:piiso}
A map $f \co F \to G$ in $\Sigma_n\circlearrowleft \wcal \Top$
is said to be a \textbf{$\pi_*$-isomorphism} if
$f$ induces a $\pi_*$-isomorphism on the underlying
spectra of $F$ and $G$.
\end{definition}

We then have the following $\Sigma_n$-equivariant analogue of
Mandell et al.\ \cite[Theorem 9.2]{mmss01}, which we state without proof.
Note that we are using the absolute stable model structure
of \cite[Section 17]{mmss01}.

\begin{lemma} \label{lem:stabmaps}
There is a stable model structure on $\Sigma_n\circlearrowleft \wcal \Top$.
It is formed by left Bousfield localising the projective model structure
at the set of maps below
\[
\{
\wcal( A \smashprod S^1,-) \smashprod S^1
\longrightarrow
\smashprod \wcal( A,-)
\  |  \ A \in \skel \wcal  \} \\
\]
The cofibrations are the same as for the projective model structure
and the weak equivalences are the $\pi_*$-isomorphisms.
This model structure is cofibrantly generated, proper and topological.
We denote it by $\Sigma_n \circlearrowleft \wcal \Sp$.
\end{lemma}

%
\subsection{Definition \texorpdfstring{of $\devcat$}{} and the projective model structure}\label{subsec:devcat}
%

We are interested in the functor $\diff_n$, which
is defined in terms of maps out of the functor
$\wcal_n (X, -)$ (Definition \ref{def:Wnjet}).
To help us study $\diff_n$ we construct
a category $\devcat$ where the $\wcal_n (X, -)$ are the representable functors.
In Section \ref{sec:diffquillen} we
show how $\diff_n$ takes values in this category.
We also note that the stable model structure on
$\devcat$ is very similar to the constructions of equivariant orthogonal spectra
in Mandell and May \cite{mm02}. In Section \ref{sec:compare} we will show it is
Quillen equivalent to spectra with a $\Sigma_n$-action.

The motivation for this construction was the classification of $n$-homogenous functors in orthogonal calculus done by Barnes and Oman \cite{barnesoman13}, by means of the category
 $O(n) \ecal_n$ (which in our current notation is $\orthdevcat$).
This category is a variation of the usual
model structure on orthogonal spectra with an $O(n)$ action; the model structure of
Mandell et al.\ \cite{mmss01} on orthogonal spectra transferred over the functor $O(n)_+\smsh -$.

\begin{definition}\label{def:Wnjet}
Let $\wcal_n$ be the category enriched over topological spaces with $\Sigma_n$-action
whose objects are those of $\wcal$ and whose spaces of morphisms are given by
\[
\wcal_n (X,Y) := \bigsmashprod{i=1}{n} \wcal (X, Y)
\]
with the $\Sigma_n$-action which permutes the factors.
(Note that this differs from the wreath product of
Definition \ref{def:wreath}.)
\end{definition}

\begin{definition}
The category $\devcat$  is the category of $\Sigma_n\circlearrowleft  \Top$-enriched functors from
$\wcal_n$ to $\Sigma_n \circlearrowleft \Top$.
\end{definition}

A functor $X$ in the category $\devcat$ consists of the following information:
a collection of based $\Sigma_n$-spaces $X(A)$
for each $A \in \wcal_n$ and a collection of  $\Sigma_n$-equivariant maps
of based  $\Sigma_n$-spaces
\[
X_{A,B} \co \wcal_n(A,B) \longrightarrow \Top (X(A), X(B))
\]

for each pair $A$, $B$ in $\wcal_n$.
The $\Sigma_n$-structure on $\Top(X(A), X(B))$ is given by conjugation.
The maps $X_{A,B}$ must be compatible with composition and also associative and unital.
They induce a structure map, where $\Sigma_n$ acts diagonally on the smash product:
\[
X(A) \smashprod \wcal_n(A,B) \longrightarrow X(B).
\]
Note that when $n=1$, $\devcat$ is just $\wcal\Top$.

We present the following without proof, as it is basically that of \cite[Theorem 6.5]{mmss01}.

\begin{theorem}\label{thm:devcatprojmod}
$\devcat$ has a projective model structure, starting with the free model structure on $\Sigma_n$-spaces. The generating cofibrations and trivial cofibrations are
\[
\begin{array}{rcl}
I_{\wcal_n} &=&\{ \wcal_n (A,-) \smashprod (\Sigma_n)_+ \smashprod i \ | \ i \in I_{\Top}, \ A\in \skel \wcal\}\\
J_{\wcal_n} &=&\{ \wcal_n (A,-) \smashprod (\Sigma_n)_+ \smashprod j \ | \ j \in J_{\Top}, \ A\in \skel \wcal \}.
\end{array}
\]
This defines a compactly generated topological proper model category
denoted $\devcat_{\proj}$.
\end{theorem}

\subsection{The stable equivalences}

We will equip $\devcat$ with a stable model structure. To do so,
we must define the weak equivalences.
Compare the following with Barnes and Oman \cite[Definition 7.7]{barnesoman13} and
Definitions \ref{def:underspec} and \ref{def:piiso}.

\begin{definition}\label{def:npistardevcat}The $n$--homotopy groups of an object $F$ of $\devcat$ at $A$ are
denoted $n \pi_p^A(F)$ and defined as
$n \pi_p^A(F) := \colim_{k \in\mathbb{Z}} \pi_p (\Omega^{nk} F(A \smashprod S^k) )
\cong \colim_{k \in\mathbb{Z}} \pi_{p+nk} (F(A \smashprod S^k) ).$
The maps of this colimit diagram are induced by adjoints of the structure maps of $F$
\[
F(A \smashprod S^k) \smashprod S^n
=
F(A \smashprod S^k) \smashprod \wcal_n(S^0, S^1)
\lra F(A \smashprod S^{k+1}).
\]
A map is said to be an \textbf{$n\pi_\ast^A$-isomorphism} if it induces isomorphisms on $n \pi_p^A$ for all $p \in \Z$.
\end{definition}

We establish independence of choice of space $A$ via Proposition \ref{prop:pistar},  which follows by the same arguments as in Mandell et al.\ \cite[Proposition 17.6]{mmss01}, so we omit the proof.
Consequently, we may speak of \textbf{$n$--homotopy groups} and \textbf{$n\pi_\ast$-isomorphisms}
without reference to a choice of space $A$.
\begin{proposition} \label{prop:pistar}
A map $f \co F \to G$ in $\devcat$ is
an $n \pi_*^A$-isomorphism for $A=S^0$ if and only if
it is an $n \pi_*^A$-isomorphism for all $A \in \wcal$. We therefore call an $n\pi_\ast^{S^0}$ isomorphism an \textbf{$n\pi_\ast$-isomorphism}.
\end{proposition}

The following Corollary is analogous to  \cite[Lemma 8.6]{mmss01}.
This says that our $n$-stable equivalences are in particular $n \pi_*$-isomorphisms.

\begin{corollary}\label{cor:lambdas}
The generalised evaluation maps
\[
\lambda_{A,n} \co \wcal_n(A \smashprod S^1, -) \smashprod S^n \longrightarrow \wcal_n(A, -)
\]
are $n \pi_*$-isomorphisms, as are the
morphisms $(\Sigma_n)_+ \smashprod \lambda_{A,n} $.
\end{corollary}

\begin{proof}
This follows from verifying that the following map is an isomorphism
\[
\colim_{k \in \mathbb{Z}} \pi_{p+nk}  (\Sigma^n \Omega^n \wcal_n (A,S^k) )
\lra
\colim_{k \in \mathbb{Z}} \pi_{p+nk} (\wcal_n (A,S^k) ).
\]
This is simply an $n$-fold version of the
$\pi_*$-isomorphism $\Sigma \Omega X \ra X$ for $X$ a spectrum.
\end{proof}

\subsection{The stable model structure}\label{subsec:fugstable}

The stable model structure on $\devcat$ is
the left Bousfield localisation
of the projective model structure at the set of maps
\begin{equation}\label{eq:lambda}
\lambda_{A,n} \co
\wcal_n(A \smashprod S^1, -) \smashprod S^n
\longrightarrow
\wcal_n(A , -)
\end{equation}
where $S^n$, viewed as $S^1 \smsh \cdots \smsh S^1$, and $\wcal_n(A,-)$ have the $\Sigma_n$-action which permutes factors.
Smash products are equipped with the diagonal action.

\begin{proposition}\label{prop:devcatmodel}
The category $\devcat$ has a stable and proper model structure with cofibrations the projective cofibrations and whose weak equivalences are the $n\pi_*$-isomorphisms (of Definition \ref{def:npistardevcat}). This model structure is denoted $\devcat_{\stable}$.

Analogous to \cite[Proposition 9.5]{mmss01}, the
fibrations are objectwise fibrations
such that the square below is a homotopy pullback.
\[
\xymatrix{
F(A) \ar[r] \ar[d] &
\Omega^n F(A \smashprod S^1)
\ar[d] \\
G(A) \ar[r]  &
\Omega^n G(A \smashprod S^1)  \\
}
\]

The fibrant objects
are the those $F$ such that the maps
$F(A) \to \Omega^n F(A \smashprod S^1)$
are weak homotopy equivalences for all $A \in \wcal$.
An $n \pi_*$-isomorphism between fibrant objects is an objectwise weak
equivalence.

The generating cofibrations are as in Theorem \ref{thm:devcatprojmod}; the generating acyclic cofibrations differ by including
the maps of following form, constructed from equation (\ref{eq:lambda})
by taking mapping cylinders and taking the pushout product with maps of the form
$(\Sigma_n)_+ \smashprod i$ for $i \in I_{\Top}$.
\[
\{ (\Sigma_n)_+ \smashprod i) \Box 
\left( \wcal_n(A \smashprod S^1, -) \smashprod S^n
\longrightarrow M(\lambda_{A,n}) \right)
\ \mid \ A \in \skel \wcal_n, \ i \in I_{\Top}
\}
\]
The homotopy category of $\devcat_{\stable}$ is generated by the object
$(\Sigma_n)_+ \smashprod \wcal_n(S^0,-)$.
\end{proposition}

\begin{proof}
This follows by the same arguments used in both
Barnes and Oman \cite[Section 7]{barnesoman13} and
Mandell et al.\ \cite[Section 9]{mmss01},
together with Corollary \ref{cor:lambdas}
(the weak equivalences are the
$n\pi_*$-isomorphisms).
The statement about generators for the homotopy category
(see Schwede and Shipley \cite[Definition 2.1.2]{ss03stabmodcat})
follows from the isomorphism
$n\pi_* (F) \cong [(\Sigma_n)_+ \smashprod \wcal_n(S^0,-), F]_*$
and Proposition \ref{prop:pistar}.
\end{proof}

%
\section{Equivalence of the two versions of spectra}\label{sec:compare}
%

We now provide an adjunction between $\devcat$ and
$\Sigma_n \circlearrowleft\wcal \Top$, then show that it is a Quillen equivalence when both categories are equipped with their stable model structures.
\[
\xymatrix@C+1.7cm{
\devcat_{\stable}
\ar@<+1ex>[r]^{\wcal \smashprod_{\wcal_n} -} &
\Sigma_n \circlearrowleft \wcal \Sp
\ar@<+1ex>[l]^{\mu_n^*}
}
\]
We start by defining the right adjoint.

%
\subsection{The adjunction \texorpdfstring{between $\devcat$ and $\Sigma_n \lca \wcal\Sp$}{} }
%

\begin{definition}\label{def:mufunctor}
We define a $\Top$-enriched functor $\mu_n : \wcal_n \ra \wcal$. 
It sends the object $X$ to  $X^{\smsh n}$ and on morphisms acts as the smash product.
It is the adjoint to $n$-fold evaluation:
\[
\wcal_n (X,Y) \smashprod X^{\smsh n} \lra Y^{\smsh n}.
\]
\end{definition}

This map of enriched categories $\mu_n$ induces a functor
$\mu_n^*$, which is (almost)  pre-composition with $\mu_n$.

Let $F$ be an object of $\Sigma_n\circlearrowleft \wcal \Top$.
Then we define $(\mu_n^\ast F)(X) = F(X^{\smsh n})$, but with an
altered action of $\Sigma_n$.
The space $F(X^{\smsh n})$ has an action of $\Sigma_n$
by virtue of $F$ being a functor to $\Sigma_n$-spaces.
We denote this action by $\sigma \mapsto \sigma^F(X^{\smsh n})$
and refer to it as the \textbf{external action}.
The space $X^{\smsh n}$ also has an action of $\Sigma_n$, denoted
$\sigma_X$ for $\sigma \in \Sigma_n$. We thus have a second action on
$F(X^{\smsh n})$, the \textbf{internal action}.
We combine these and define the action on
$(\mu_n^\ast F)(X)$ to be %
$\sigma \in \Sigma_n \mapsto \sigma^F(X) \cdot F(\sigma_X).$
Note that the internal and external actions commute.

We complete our definition of $\mu_n^* F$ by giving its structure map below,
where $\nu^F$ is the structure map of $F$.
\vskip-0.5cm
\[
\xymatrix@C+1cm{
\wcal_n(X,Y) \smashprod F(X^{\smsh n})
\ar[r]^-{(\mu_n)_{X,Y} \smashprod \id} &
\wcal(X^{\smsh n},Y^{\smsh n}) \smashprod F(X^{\smsh n})
\ar[r]^-{\nu^F_{X^{\smsh n},Y^{\smsh n}}} &
F(Y^{\smsh n}).
}
\]
We must now show that this map is $\Sigma_n$-equivariant
using the altered action on $F(X^{\smsh n})$ and $F(Y^{\smsh n})$
and the permutation action on $\wcal_n(X,Y)$.
The action on $\wcal(X^{\smsh n},Y^{\smsh n})$ is via conjugation:
$f \mapsto \sigma_Y \circ f \circ \sigma_X^{-1}$.
The first map is clearly $\Sigma_n$-equivariant.
For the second map, we look at the actions separately.
By naturality of $\nu^F$, the second map is equivariant with respect
to the internal actions on $F(X^{\smsh n})$ and $F(Y^{\smsh n})$
and the action on $\wcal(X^{\smsh n},Y^{\smsh n})$.
It is also equivariant with respect to the external actions
on $F(X^{\smsh n})$ and $F(Y^{\smsh n})$ (with no action on
$\wcal(X^{\smsh n},Y^{\smsh n})$). Composing the two actions gives the result.

\begin{remark}\label{rmk:equivariance}
We compare the different versions of equivariance for
$\devcat$ and $\Sigma_n\circlearrowleft\Top$.
Consider some $F: \wcal \ra \Sigma_n\circlearrowleft\Top$. Then $F(A) \in \Sigma_n \circlearrowleft \Top$ and for a map $f \in \wcal (A,B)$, the map $F(f): F(A) \ra F(B)$ is $\Sigma_n$-equivariant.
That is, $F$ induces a map
\[F_{A,B}: \wcal (A,B)\ra  \Top (F(A), F(B))^{\Sigma_n}.\]
In contrast, for $G \in \devcat$,
the following is a $\Sigma_n$-equivariant map.
\[
G_{A,B}:  \wcal (A,B)^{\smsh n} \ra  \Top (G(A), G(B))
\]
The functor $\mu_n^*$ allows us to compare these two types
of equivariance. Indeed, the altered action on $(\mu_n^* F)(X)$
is designed precisely to take account of
the non-trivial $\Sigma_n$-action on $W_n(X,Y)$.
\end{remark}

The left adjoint $\wcal \smashprod_{\wcal_n} -$ takes an object $F$ of $\devcat$ to the coend
\[
\int^{A \in \wcal_n} F(A) \smashprod \wcal(A^{\smsh n},-).
\]
The term $\wcal(A^{\smsh n},-)$ has an action of $\Sigma_n$
by permuting the factors of $A^{\smsh n}$.
Establishing the adjunction is a formal exercise in manipulating ends and coends.

%
\subsection{The Quillen equivalence}\label{subsec:gooddump}
%

In this section we prove that the adjunction we have established is a Quillen equivalence.
\[
\xymatrix@C+1.7cm{
\devcat_{\stable}
\ar@<+1ex>[r]^{\wcal \smashprod_{\wcal_n} -} &
\Sigma_n \circlearrowleft \wcal \Sp
\ar@<+1ex>[l]^{\mu_n^*}
}
\]

\begin{lemma}
The adjoint pair $(\wcal \smashprod_{\wcal_n} -, \mu_n^*)$
is a Quillen pair with respect to the stable model structures.
\end{lemma}

\begin{proof}
A generating cofibration of $\devcat$ is of the form
$\wcal_n(A,-) \smashprod (\Sigma_n)_+ \smashprod i$,
for $i$ a cofibration of based spaces.
The left adjoint sends this map
to $\wcal(A^{\smsh n}, -)\smashprod (\Sigma_n)_+ \smashprod i$,
which is a cofibration of $\Sigma_n \circlearrowleft \wcal \Sp$.
Similarly, it sends the generating acyclic cofibrations
of the projective model structure on $\devcat$ to
acyclic cofibrations of
$\Sigma_n \circlearrowleft \wcal \Sp$.

The stable model structure on $\devcat$
comes from taking the projective model structure
and localising at the maps
\[
\wcal_n(A \smashprod S^1, -)
\smashprod
S^n \to \wcal_n(A,-)
\]
The left adjoint will take a map of the form above to the $\pi_*$-isomorphism
\[
\wcal(A^{\smsh n} \smashprod S^n, -)
\smashprod S^n \to
\wcal(A^{\smsh n},-).
\]
It follows that the left adjoint is a left Quillen functor.
\end{proof}

\begin{proposition}\label{prop:phiQE}
The adjoint pair $(\wcal \smashprod_{\wcal_n} -, \mu_n^*)$
is a Quillen equivalence.
\end{proposition}

\begin{proof}
We claim that the right adjoint preserves all weak equivalences.
A map $f$ is a weak equivalence of
$\Sigma_n \circlearrowleft \wcal \Sp$
if and only if
$f[S^0]$ is a $\pi_*$-isomorphism of spectra by Definition \ref{def:piiso}.
Similarly a map $g$ is a weak equivalence of
$\devcat$ if and only if it is an $n\pi_\ast$-isomorphism,
by Proposition \ref{prop:devcatmodel}. By Proposition \ref{prop:pistar}, $g$ is an $n\pi_\ast$-iso if and only if $n\pi_\ast^{S^0}(g)\co n\pi_\ast^{S^0}(F) \ra n\pi_\ast^{S^0}(G)$ is an isomorphism.

Consider $\mu_n^* F$ for some object $F$ in
$\Sigma_n \circlearrowleft \wcal \Sp$.
It is routine to check that
\[
n\pi^{S^0}_p (\mu_n^\ast F) = \colim_{k \in\mathbb{Z}} \pi_{p+nk} F(S^{nk}).
\]
By cofinality of the terms $p+nk$ in $\mathbb{Z}$, if follows that $\mu_n^\ast f$ is an $n \pi_*$-isomorphism whenever $f[S^0]$ is a $\pi_*$-isomorphism.
Hence we have shown our claim that the right adjoint preserves all weak equivalences.

By \cite[Corollary 1.3.16]{hov99}, we must now show that for cofibrant $F \in \devcat$, the derived unit map of the adjunction is a weak equivalence. Since  the right adjoint preserves all weak equivalences (in particular, that between an object and its fibrant replacement), it is enough to consider the unit map
\[
F \longrightarrow \mu_n^* \wcal \smashprod_{\wcal_n} F.
\]
By stability, it suffices to check this in the case of the single generator of
the homotopy category of $\devcat$.
Replacing $F$ by this generator and simplifying,
we are left with the map below, which is induced by $\mu_n$.
\[
(\Sigma_n)_+ \smashprod  \wcal(S^0, -)^{\smashprod n}
\longrightarrow
(\Sigma_n)_+ \smashprod  \wcal(\mu_n(S^0),\mu_n(-))
=
(\Sigma_n)_+ \smashprod  \wcal(S^0,(-)^{\smashprod n})
\]
This map is an isomorphism, hence it is a weak equivalence as desired.
\end{proof}

%
\section{Differentiation is a Quillen equivalence}\label{sec:diffquillen}
%

In this section we define differentiation
as an adjunction between the homogeneous model structure on
$\wcal \Top$ and the stable model structure on $\devcat$.
We then show that it is a Quillen equivalence.
Thus we will have a diagram of Quillen equivalences as below,
showing that ${\wcal \Top_{\nhomog}}$ is Quillen equivalent to
spectra with a $\Sigma_n$--action. Finally we will show that this
diagram captures precisely Goodwillie's classification theorem.
\[
\xymatrix@C+1.5cm@R+1.2cm{
{\wcal \Top_{\nhomog}}
\ar@<-3pt>[r]_-{\diff_n}
&
\devcat_{\stable}
\ar@<3pt>[r]^-{\wcal \smashprod_{\wcal_n} -}
\ar@<-3pt>[l]_-{(-)/\Sigma_n \circ \mapdiag^\ast}
&
\ar@<3pt>[l]^-{\mu_n^\ast}
\Sigma_n\circlearrowleft \wcal Sp
\\
}
\]

\subsection{The adjunction \texorpdfstring{between $\devcat$ and $\wcal \Top_{\nhomog}$}{} }\label{subsec:diffnquillen}

Recall Definition \ref{def:cref} where we define the $n^{th}$-cross effect. The
$n^{th}$-derivative of $F$ is
\[
\diff_n (F)(X) = \Nat (\bigsmashprod{l=1}{n} \wcal(X,-) , F)
\]
which is cross effect pre-composed with the diagonal, which we originally considered as an object of $\wcal\Top$.
The category $\devcat$ is the most natural
target for the functor $\diff_n,$
as its representable functors are of the form
$\wcal_n (X, -)=\bigsmashprod{l=1}{n} \wcal(X,Y)$.
\begin{definition}
We define the $n^{th}$ \textbf{(stable) derivative} of $F \in \wcal \Top$,
to be the functor $\diff_n (F)$ in $\devcat$.
The structure map below is induced from the composition of $\wcal_n$
and is $\Sigma_n$-equivariant
\[
\wcal_n(X,Y) \smashprod
\Nat (\wcal_n(X,-) , F) \lra
\Nat (\wcal_n(X,-) , F).
\]
\end{definition}
 
\begin{proposition}\label{prop:diffnleft}
The functor $\diff_n$ has a left adjoint:
\[
{(-)/\Sigma_n \circ \mapdiag^\ast} \co \devcat
\longrightarrow
\wcal \Top
\]
which we define in the proof below.
\end{proposition}

\begin{proof}
We begin by defining the $\Top$-enriched functor
$\mapdiag: \wcal \ra \wcal_n$. It is the identity on
objects and the diagonal on morphisms:
\[
f \in \wcal (A, B) \mapsto [(f, \ldots, f)] \in  \wcal(A,B)^{\smsh n} = \wcal_n (A,B).
\]
In particular, $\mapdiag$ lands in the $\Sigma_n$-fixed points of
$\wcal_n(A,B)$.
Let $E \in \devcat$, then for $X \in \wcal$, $E(X)$ is a space
with an action of $\Sigma_n$. We use a shorthand
\[
E(X)/\Sigma_n: = ({(-)/\Sigma_n \circ \mapdiag^\ast}(E))(X)
\]
We must also describe the structure maps of
$E(-)/\Sigma_n \in \wcal \Top$.
Consider the composite
\[
E(X) \smashprod \wcal(X,Y)
\lra
E(X) \smashprod \wcal_n(X,Y)
\lra
E(Y)
\]
where the first map is $\id \smashprod \mapdiag$ as defined above and the second is
the structure map of $E \in \devcat$. If we equip $\wcal(X,Y)$ with the trivial action,
then this composite is $\Sigma_n$-equivariant. Hence, we can
apply $(-)/\Sigma_n$ to this map, the result of which is
the structure map of $E(-)/\Sigma_n \in \wcal \Top$.

A gentle exercise in category theory shows that we have an adjunction:
\[
\begin{array}{rcl}
\wcal \Top({E/\Sigma_n \circ \mapdiag}, F)
&=&
\int_{X \in \wcal} \Top (E(X)/\Sigma_n, F(X))  \\
 &\cong&
\int_{Y \in \wcal_n} \Sigma_n \Top \left(
E(Y) , \diff_n (F)(Y) \right) \\
&=&
\devcat \left(E , \diff_n (F) \right).
\end{array}
\]
\end{proof}

\begin{lemma}\label{lem:adjuntop}
The adjunction $((-)/\Sigma_n \circ \mapdiag^*,\diff_n)$ is $\Top$-enriched.
\end{lemma}

\begin{proof}
There is an isomorphism, natural in $E \in \devcat$ and
$K \in \Top$
\[
(({(-)/\Sigma_n \circ \mapdiag^*})(E)) \smashprod K
\to
({(-)/\Sigma_n \circ \mapdiag^*})(E \smashprod K))
\]
induced by the isomorphism
$E(X)/\Sigma_n \smashprod K \to (E(X)\smashprod K)/\Sigma_n$.
It follows that the right adjoint commutes with the cotensoring with $\Top$
and that the adjunction is enriched over topological spaces.
\end{proof}

With the language of parameterised spectra
(in the sense of May and Sigurdsson \cite{ms06})
we could extend our definitions of $\devcat$ and $\diff_n$
to capture derivatives of functors of spaces over $Y$. As is common in the modern literature, %
 we concentrate on the fundamental case of the derivative over a point.

\subsection{The Quillen equivalence}

\begin{proposition}\label{prop:diffquillen}
The adjunction $((-)/\Sigma_n \circ \mapdiag^*,\diff_n)$ is a
Quillen pair with respect to the following pairs of model structures.
\begin{enumerate}[noitemsep]
\item $\devcat_{\proj}$ and $\wcal \Top_{\cross}$.
\item $\devcat_{\proj}$ and $\wcal \Top_{\nexs}$.
\item $\devcat_{\stable}$ and $\wcal \Top_{\nexs}$.
\item $\devcat_{\stable}$ and $\wcal \Top_{\nhomog}$.
\end{enumerate}
\end{proposition}
%
\begin{proof}
A generating (acyclic) cofibration of the projective model structure on
$\devcat$ has the form
$\wcal_n (A,-) \smashprod (\Sigma_n)_+ \smashprod i$,
where $i$ is
a generating (acyclic) cofibration for based spaces.
By Lemma \ref{lem:adjuntop} the functor ${(-)/\Sigma_n \circ \mapdiag^*}$
takes this to the map
$\wcal_n (A,-) \smashprod i$
of $\wcal \Top$, which is a (acyclic) cofibration of the cross effect model structure on $\wcal \Top$ by Corollary \ref{cor:wncofibrations}.
Thus ${(-)/\Sigma_n \circ \mapdiag^*}$ is a left Quillen
functor as claimed in Part (1.).

Part (2.) holds as every (acyclic) cofibration of
the cross effect model structure on $\wcal \Top$ is a
(acyclic) cofibration of the $n$--excisive model structure on $\wcal \Top$.

For Part (3.), by Hirschhorn \cite[Theorem 3.1.6]{hir03}
we only need to show that
$\diff_n$ takes fibrant objects of the $n$-excisive model structure
to fibrant objects of the stable model structure.
That is, if $F$ is $n$-excisive and cross effect fibrant,
then for any $A \in \wcal$
\[
(\diff_n F)(A) \to \Omega^n (\diff_n F)(A \smashprod S^1)
\]
is a weak homotopy equivalence.
This is the content of Goodwillie \cite[Proposition 3.3]{goodcalc3} with the assumption that $F(\ast)$ is equal to $\ast$, rather than just weakly equivalent. This assumption holds true for any object of $\wcal \Top$ as is noted in Section \ref{subsec:catwtop}.

For Part (4.), the cofibrations of the $n$-stable model structure have the form
$\wcal_n (A,-) \smashprod (\Sigma_n)_+ \smashprod i.$
Such a map is sent by $(-)/\Sigma_n \circ \mapdiag^*$ to  $\wcal_n (A,-) \smashprod i$,
which is a
cofibration of the $n$-excisive model structure
by Corollary \ref{cor:wncofibrations}.
This map is a cofibration of
the $n$--homogeneous model structure by
\cite[Proposition 3.3.16]{hir03} and \cite[Lemma 5.5.2]{hir03}.
So the left adjoint preserves cofibrations.
The acyclic cofibrations are the same as in Part (3),
hence the left adjoint preserves acyclic cofibrations.
\end{proof}

The following lemma provides an even simpler description of
the weak equivalences of the $n$-homogeneous model structure.
That is, one only has to know how to calculate the spaces
$\diff_n \fibrep_{\cross} F(X)$
(which is the homotopy cross effect precomposed with the diagonal)
to understand the behaviour of $\hodiff_n F$.  There is no need
to apply $P_n$ as we are no longer interested
in the objectwise weak homotopy equivalences, but the $n \pi_*$-isomorphisms.
This justifies the use of \emph{stable} when calling $\diff_n F \in \devcat$ the $n^{th}$ stable derivative.

\begin{lemma}\label{lem:stability}
Let $f \co F \to G$ be a map of
cross-effect fibrant functors.
If $f$ is a weak equivalence in the $n$-homogeneous model structure on $\wcal \Top$, then $\diff_n f$ is an
 $n \pi_*$-isomorphism.
In particular, the weak equivalences of $\wcal \Top_{\nhomog}$ are those maps
$f$ such that $\diff_n \fibrep_{\cross} f$ is an $n \pi_*$-isomorphism.
\end{lemma}
%
\begin{proof}
Let $F \in \wcal \Top$, then Goodwillie's functor $T_{1, \dots, 1}$, which is $\T_1$ applied in each variable, applied to $\cref_nF$ can be written as
\[
\scalebox{1}{$
(\T_{1,\dots,1} \cref_n F)(A, \ldots, A)
\simeq
 \Omega^n (\cref_n F)(A \smashprod S^1, \ldots, A \smashprod S^1)
= \Omega^n (\diff_n F)(A \smashprod S^1)
$}
\]
Abusing notation, we will write
$\T_{1,\ldots,1} \diff_n F (A)$ for
$\Omega^n (\diff_n F)(A \smashprod S^1)$,
even though $\diff_n$ is not an $n$-variable functor.
Recall that the functor $P_{1,\dots,1}$ is the homotopy colimit
of repeated applications of $T_{1,\dots,1}$.
Hence the map $\diff_n F \to P_{1,\ldots, 1} \diff_n F$
is (weakly equivalent to) the fibrant replacement functor
of the stable model structure on $\devcat$.
Consequently, the maps $\alpha$ and $\delta$ below are stable equivalences.
\[
\xymatrix{
\diff_n \fibrep_{\cross} F \ar[r]^-{\alpha} \ar[d]^{\beta} &
P_{1,\ldots, 1} \diff_n \fibrep_{\cross} F \ar[d]^{\gamma} \\
\diff_n \fibrep_{\cross} P_n F  \ar[r]^-{\delta} &
P_{1,\ldots, 1} \diff_n \fibrep_{\cross} P_n F \\
}
\]
The map $\gamma$
is an objectwise homotopy weak equivalence by Biedermann and R\"ondigs
\cite[Theorem 5.35]{BRgoodwillie}, when viewed as a map
\[
P_{1,\ldots,1} \cref_n \fibrep_{\cross} F \ra P_{1,\ldots,1} \cref_n \fibrep_{\cross} P_n F.
\]
So $\beta$ is an $n \pi_*$-isomorphism and the result follows immediately.
\end{proof}

\begin{theorem}\label{thm:diffQE}
The adjunction $((-)/\Sigma_n \circ \mapdiag^*,\diff_n)$ is a
Quillen equivalence with respect to the $n$--stable model structure on
$\devcat$ and the $n$--homogeneous model structure on $\wcal \Top$.
\end{theorem}
%
\begin{proof}
We show that the right adjoint reflects weak equivalences between fibrant objects. Let $g \co X \to Y$ be a map between cross effect fibrant $n$-excisive functors in $\wcal \Top$ such that
$\diff_n g$ is an $n \pi_*$-isomorphism in $\devcat$.
The domain and codomain of $\diff_n g$ are fibrant in
the stable model structure by
Proposition \ref{prop:diffquillen}.
Hence, $\diff_n g$ is an objectwise weak homotopy equivalence
by Proposition \ref{prop:devcatmodel}.
The fibrancy assumption also tells us that $g \simeq \fibrep_{\nexs} g$,
so that $\diff_n g$ is weakly equivalent to $\hodiff_n g$.
Thus $\diff_n g$ is a
weak equivalence of the $n$--homogeneous model structure.

We now show that for any cofibrant $E \in \devcat$,
the derived unit map
\[
E \lra \hodiff_n ((-)/\Sigma_n \circ \mapdiag^*(E))
\]
is an $n \pi_*$-isomorphism.
But this derived unit map is the map $\theta$ of Goodwillie \cite[Theorem 3.5]{goodcalc3},
which is an equivalence.
Thus by Hovey \cite[Corollary 1.3.16]{hov99} this adjunction is a Quillen
equivalence.
\end{proof}

Recall the functor $\mu_n$ of Definition \ref{def:mufunctor} and
$\mu_n^* \co \Sigma_n \lca \wcal \Sp \to \devcat_{\stable}$,
the right adjoint of the Quillen equivalence of Proposition \ref{prop:phiQE}.

\begin{theorem}\label{thm:correctdervied}
The composite of the derived
functors of $\mapdiag^* \circ (-)/\Sigma_n$ and $\mu_n^*$
agrees with Goodwillie's classification of
$n$-homogeneous functors
(recall Figure \ref{fig:homogclass}, see \cite[Section 2-5]{goodcalc3}).
That is, for a $\Sigma_n$-spectrum, $D$, we have that
$(\mathbb{L} \mapdiag^* \circ (-)/\Sigma_n \circ \mathbb{R} \mu_n^* )(D)$
is weakly equivalent in $\wcal \Top_{\nhomog}$ to the functor
\[
A \mapsto \Omega^\infty \big( (D \smashprod A^{\smashprod n})/h \Sigma_n \big).
\]
\end{theorem}
%
\begin{proof}
The functor $\mapdiag^* \circ (-)/\Sigma_n$ preserves
objectwise weak homotopy equivalences (such as acyclic fibrations
in the stable model structure) between $\Sigma_n$-free objects.
Hence the derived composite of $\mapdiag^* \circ (-)/\Sigma_n$ and $\mu_n^*$
applied to some $E \in \Sigma_n \lca \Sp$ is given by
\begin{equation}
A  \mapsto  (E \Sigma_n)+ \smashprod_{\Sigma_n}
\big( \hocolim_{k \in\mathbb{Z}} \Omega^k E(A^{\smashprod n} \smashprod S^k) \big) \tag{$a$}
\end{equation}
We claim that the functor $(a)$ is weakly equivalent in $\wcal \Top_{\nhomog}$ to each of the following two functors
\begin{align}
A  \mapsto  (E \Sigma_n)+ \smashprod_{\Sigma_n}
\big( \Omega^\infty (E \smashprod B^{\smashprod n}) \big) \tag{$b$} \\
A  \mapsto
\Omega^\infty \big(  (E \Sigma_n)+ \smashprod_{\Sigma_n} (E \smashprod B^{\smashprod n}) \big) \tag{$c$}.
\end{align}
That $(a)$ is equivalent to $(b)$ follows from
Mandell et al.\ \cite[Proposition 17.6]{mmss01}, which implies
that there is a natural weak homotopy equivalence of spaces
\[
\Omega^\infty (E \smashprod B^{\smashprod n}) :=  \hocolim_{k \in\mathbb{Z}} \Omega^k (E(S^k) \smashprod B^{\smashprod n})
\lra
\hocolim_{k \in\mathbb{Z}} \Omega^k E(B^{\smashprod n} \smashprod S^k)
\]
By the connectivity
arguments of Weiss \cite[Example 6.4]{weiss95},
the functors $(b)$ and $(c)$ agree up to order $n$,
in the sense of \cite[Definition 1.2]{goodcalc3}.
By \cite[Proposition 1.6]{goodcalc3} they are $P_n$-equivalent.
Hence our derived functor is weakly equivalent in $\wcal \Top_{\nhomog}$
to Goodwillie's formula.
\end{proof}

Note that our derived composite sends a spectrum $E$
to a functor in $\wcal \Top_{\nhomog}$ which is weakly equivalent to that of Goodwillie's theorem.
However the formula of Goodwillie actually creates an $n$-homogenous functor directly.
%
\begin{ex}\label{ex:deriv1}
As an example of how our version of the classification can
make calculations easier,
consider the cofibrant object
$(\Sigma_n)_+ \smashprod \wcal_n(X,-)$ of $\devcat$.
The derived functor of $\wcal \smashprod_{\wcal_n} -$
sends this to $(\Sigma_n)_+ \smashprod \wcal(X^{\smashprod n},-)$.
Equally the derived functor of
$(-)/\Sigma_n \circ \mapdiag^\ast$ sends this to
$\wcal_n(X,-)$ in $\wcal \Top_{\nhomog}$.
Hence we have that the $n$-homogeneous part of
$\wcal_n(X,-)$ is classified by the $\Sigma_n$--spectrum
$(\Sigma_n)_+ \smashprod \wcal(X^{\smashprod n},-)$.
In the case $X=S^0$, this says that the functor $A \to A^{\smashprod n}$
has $n^{th}$ derivative $(\Sigma_n)_+ \smashprod \sphspec$,
(recall the sphere spectrum in $\wcal \Sp$ is
given by $\wcal(S^0,-)$).
This is analogous to the statement that
the $n^{th}$ derivative of $x^n$ is $n!$.
\end{ex}

\begin{ex}\label{ex:deriv2}
We may also make an analogy to the statement:
the $n^{th}$ derivative of $x^n/n!$ is $1$.
Consider the non-cofibrant object
$\wcal_n(X,-)$ of $\devcat$.
The derived functor of $\wcal \smashprod_{\wcal_n} -$
sends this to $\wcal(X^{\smashprod n},-)$.
Equally the derived functor of
$(-)/\Sigma_n \circ \mapdiag^\ast$ sends $\wcal_n(X,-)$ to
$(E \Sigma_n)_+ \smashprod_{\Sigma_n} \wcal_n(X,-)$ in $\wcal \Top_{\nhomog}$.
In the case $X=S^0$, this says that the functor
$A \to A^{\smashprod n}/h \Sigma_n$
has $n^{th}$ derivative given by $\sphspec$.
\end{ex}

In general, we can take a spectrum with $\Sigma_n$-action, find a model
for it in $\devcat$ and then easily calculate its image in $\wcal \Top_{\nhomog}$.
Finding a model for a spectrum with $\Sigma_n$-action in $\devcat$ is a standard problem, akin to finding a nice point-set model of an EKMM spectrum in terms of orthogonal spectra or symmetric spectra. This combined with Lemma \ref{lem:stability}
shows how our new perspective and description of the classification simplifies some
calculations.

\section{Quillen equivalence with symmetric multilinear functors}\label{sec:comp-symm-lin}
We establish in Theorem \ref{thm:devcat-BR-eq} a Quillen equivalence between our new category  $\devcat_{\stable}$ and $\symfun( \wcal^n, \Top)_{ml}$,
the category of symmetric functors with the symmetric-multilinear model structure. This result makes it clearer still that the category of symmetric multilinear functors can be omitted from the classification of $n$-homogeneous functors.

We begin by giving some definitions and recalling the statement of the symmetric-multilinear model structure of Biedermann and R\"ondigs \cite[Theorem 5.20]{BRgoodwillie}. Let $\wcal^n$ be the topological category with objects $n$--tuples of spaces in $\wcal$ and morphisms spaces $\wcal (X_1, Y_1) \smsh \cdots \smsh \wcal (X_n, Y_n)$ for $(X_1, \ldots, X_n)$ and $(Y_1, \ldots, Y_n)$ in $\wcal^n$. There are a pair of obvious $\Top$-enriched functors, $\Delta$ and $\smsh$, between this category and $\wcal$ given by
\[
\begin{array}{rclcrcl}
\Delta \co \wcal & \lra & \wcal^n &\hspace{1cm} & \smashprod \co \wcal^n & \lra & \wcal \\
X & \mapsto & (X, \ldots, X) & & (X_1, \ldots, X_n ) & \mapsto& X_1\smsh X_2 \smsh \cdots \smsh X_n\\
f:X\ra Y & \mapsto & [(f, \ldots, f)] & & [(f_1, \ldots, f_n )] & \mapsto& f_1\smsh f_2 \smsh \cdots \smsh f_n .\\
\end{array}
\]

There is a less obvious $\Top$-enriched functor from $\wcal_n$ to $\wcal^n$,
which we call $\obdiag$. It is the diagonal on objects and the identity on morphism spaces. That is,

\[
\begin{array}{rcl}
\obdiag: \wcal_n & \lra & \wcal^n \\
X & \mapsto& (X, \ldots, X) \\
\wcal_n (X,Y) = \bigsmashprod{l=1}{n} \wcal (X,Y) &
\overset{\id}{\mapsto} &
\wcal_n (X,Y) = \bigsmashprod{l=1}{n} \wcal (X,Y)\\
\end{array}
\]
Recall the functor $\mapdiag$, defined in the proof of Proposition \ref{prop:diffnleft}. The
diagonal functor $\Delta$ as given above is the composite $\obdiag \circ \mapdiag$.

Let $\symfun( \wcal^n, \Top)$ denote the category of symmetric functors from $\wcal^n$ to $\Top$. An $n$-variable functor $F$ is symmetric precisely when, for each $\sigma \in \Sigma_n$, there is a natural isomorphism $F(X_1, \ldots, X_n) \cong F(X_{\sigma(1)},\ldots, F(X_{\sigma (n)}))$. When $F$ is symmetric and $X_l=X$ for all $l$,
$F(X, \dots ,X)$ has an action of $\Sigma_n$.
Using this action and pre-composition with $\obdiag$, we obtain a functor from $\symfun( \wcal^n, \Top)$
to
$\devcat$ which we call $\obdiag^*$. We can also
consider $\cref_n$ as a functor from $\wcal \Top$
to $\symfun( \wcal^n, \Top)$.
Since the cross effect precomposed with the diagonal is
the functor $\diff_n$, we have
the following commutative diagram of functors.
\[
\xymatrix{
\devcat
& &
\wcal \Top
\ar[ll]_-{\diff_n}
\ar[dl]^-{\cref_n}
\\
&
\symfun( \wcal^n, \Top)
\ar[ul]^-{\obdiag^*}
}
\]

We use this diagram to relate our work and that
of Biedermann and R\"ondigs \cite{BRgoodwillie}.
They develop a \textbf{symmetric multilinear model structure} on  $\symfun( \wcal^n, \Top)$,
a modification of their \textit{hf}(``homotopy functor")-model structure.
In $\wcal\Top$, all of our functors are homotopy functors and the hf-model structure is then the projective model structure. We modify their statements (see \cite[Definition 5.19]{BRgoodwillie}) accordingly:

\begin{theorem}\label{thm:BRsym}
There is a model category $\symfun( \wcal^n, \Top)_{ml}$
whose underlying category is the category of symmetric functors
from $\wcal^{n}$ to $\Top$.
The weak equivalences are the maps $f$ such that $P_{1,...,1}(f)$
is an objectwise weak homotopy equivalence,
called multilinear equivalences; %
the cofibrations are the projective cofibrations; and %
the fibrations are the objectwise fibrations $f: F\ra G$ such that either (and hence both) of
the following squares
\[
\xymatrix{
F \ar[r] \ar[d]  & P_{1, \ldots, 1} F \ar[d]^{P_{1, \ldots, 1} (f)} \\
G \ar[r]           & P_{1, \ldots, 1}  G\\
}
\qquad \qquad
\xymatrix{
F \ar[r] \ar[d]  & T_{1, \ldots, 1} F \ar[d]^{T_{1, \ldots, 1} (f)} \\
G \ar[r]           & T_{1, \ldots, 1}  G\\
}
\]
is an objectwise homotopy pullback square.
Moreover, the fibrant objects are the symmetric multilinear functors.
\end{theorem}

In proving this result, it is helpful to have a different, but equivalent, description of the category.
This alternate description (adjusted to our setting)
is given below, see \cite[Lemma 3.6]{BRgoodwillie}.

\begin{definition}\label{def:wreath}
The \textbf{wreath product category} $(\Sigma_n \wr \wcal^n)$ has objects
the class of $n$-tuples $(X_1, \ldots, X_n)$ of objects of $\wcal$. The morphisms from  $\uline{X} =(X_1, \ldots, X_n)$ to $\uline{Y} =(Y_1, \ldots, Y_n)$ are given by
\[
(\Sigma_n \wr \wcal^n) \left(  \uline{X}, \uline{Y}  \right)
=
\bigvee_{\sigma \in \Sigma_n} \bigsmashprod{l=1}{n} \wcal (X_l, Y_{\sigma^{-1}(l)})
\]
with composition defined as for the wreath product of groups.
\end{definition}

Given the model structure of Theorem \ref{thm:BRsym}, we may now establish the following formal comparison of our work with that of \cite{BRgoodwillie}.
\begin{theorem}\label{thm:devcat-BR-eq} The functor $\obdiag^\ast$ is a right Quillen adjoint, and induces a Quillen equivalence
between $\symfun( \wcal^n, \Top)_{\ml}$ and $\devcat$ with the stable model structure.
\end{theorem}
%
\begin{proof}
Recall that in the stable model structure on $\devcat$
(Proposition \ref{prop:devcatmodel})
the fibrations are those maps $f:F \ra G$ which are objectwise fibrations, such that
square below is an objectwise homotopy pullback.
\[
\xymatrix{
F \ar[r] \ar[d]_f  &\Omega^n F (-\smsh S^1)  \ar[d]^{f(- \smsh S^1)}\\
G \ar[r]           & \Omega^n G(-\smsh S^1) \\
}
\]
The functor $\obdiag^*$ preserves objectwise (acyclic) fibrations.
Moreover if the right hand square of Theorem \ref{thm:BRsym}
is an objectwise pullback square, then $\obdiag^*$
sends it to a square of the same form as the above.
Therefore, $\obdiag^*$ is a right Quillen functor.

The functor $\diff_n$ is the right adjoint of a Quillen equivalence by
Theorem \ref{thm:diffQE} whereas $\cref_n$ is the right adjoint of a
Quillen equivalence by \cite[Corollary 6.17]{BRgoodwillie}.
Since $\obdiag^* \circ \cref_n = \diff_n$, it follows that
$\obdiag^*$ is also part of a Quillen equivalence.
\end{proof}

\addcontentsline{toc}{part}{Bibliography}
\bibliographystyle{plain}

\begin{thebibliography}{10}

\bibitem{aroneweiss}
G.~Arone.
\newblock The weiss derivatives of {BO (-)} and {BU (-)}.
\newblock {\em Topology}, 41(3):451--481, 2002.

\bibitem{barneseldred15}
D.~Barnes and R.~Eldred.
\newblock Comparing the orthogonal and homotopy functor calculi.
\newblock To appear, {arXiv:1505.05458},  2015.

\bibitem{barnesoman13}
D.~Barnes and P.~Oman.
\newblock Model categories for orthogonal calculus.
\newblock {\em Algebr. Geom. Topol.}, 13(2):959--999, 2013.

\bibitem{barnesroitzheimstable}
D.~Barnes and C.~Roitzheim.
\newblock Stable left and right {B}ousfield localisations.
\newblock {\em Glasg. Math. J.}, 56(1):13--42, 2014.

\bibitem{BCR07}
G.~Biedermann, B.~Chorny, and O.~R{\"o}ndigs.
\newblock Calculus of functors and model categories.
\newblock {\em Adv. Math.}, 214(1):92--115, 2007.

\bibitem{BRgoodwillie}
G.~Biedermann and O.~R{\"o}ndigs.
\newblock Calculus of functors and model categories, II.
\newblock {\em Algebr. Geom. Topol.}, 14(5):2853--2913, 2014.

\bibitem{bous01}
A.~K. Bousfield.
\newblock On the telescopic homotopy theory of spaces.
\newblock {\em Trans. Amer. Math. Soc.}, 353(6):2391--2426, 2001.

\bibitem{CI04}
J.~D. Christensen and D.~C. Isaksen.
\newblock Duality and pro-spectra.
\newblock {\em Algebr. Geom. Topol.}, 4:781--812, 2004.

\bibitem{gw90}
T.~Goodwillie.
\newblock Calculus {I}: The first derivative of pseudoisotopy theory.
\newblock {\em K-Theory}, 4(1):1--27, 1990.

\bibitem{gw91}
T.~Goodwillie.
\newblock Calculus {II}: {A}nalytic functors.
\newblock {\em K-Theory}, 5(4):295-- 332, 1991/92.

\bibitem{goodcalc3}
T.~Goodwillie.
\newblock Calculus {III}: {T}aylor {S}eries.
\newblock {\em Geom. Topol.}, 7:645--711, 2003.

\bibitem{heller82}
A.~Heller.
\newblock Homotopy in functor categories.
\newblock {\em Trans. Amer. Math. Soc.}, 272(1):185--202, 1982.

\bibitem{hir03}
P.~S. Hirschhorn.
\newblock {\em Model categories and their localizations}, volume~99 of {\em
  Mathematical Surveys and Monographs}.
\newblock American Mathematical Society, Providence, RI, 2003.

\bibitem{hov99}
M.~Hovey.
\newblock {\em Model categories}, volume~63 of {\em Mathematical Surveys and
  Monographs}.
\newblock American Mathematical Society, Providence, RI, 1999.

\bibitem{kell05}
G.~M. Kelly.
\newblock Basic concepts of enriched category theory.
\newblock {\em Reprints in Theory and Applications of Categories}, (10), 2005.
\newblock Reprint of the 1982 original [Cambridge Univ. Press, Cambridge;
  MR0651714].

\bibitem{luriehigher}
J.~Lurie.
\newblock Higher algebra.
\newblock http://www. math. harvard. edu/\~{} lurie, 2014.

\bibitem{mm02}
M.~A. Mandell and J.~P. May.
\newblock Equivariant orthogonal spectra and {$S$}-modules.
\newblock {\em Mem. Amer. Math. Soc.}, 159(755):x+108, 2002.

\bibitem{mmss01}
M.~A. Mandell, J.~P. May, S.~Schwede, and B.~Shipley.
\newblock Model categories of diagram spectra.
\newblock {\em Proc. London Math. Soc. (3)}, 82(2):441--512, 2001.

\bibitem{mp12}
J.~P. May and K.~Ponto.
\newblock {\em More concise algebraic topology}.
\newblock Chicago Lectures in Mathematics. University of Chicago Press,
  Chicago, IL, 2012.

\bibitem{ms06}
J.~P. May and J.~Sigurdsson.
\newblock {\em Parametrized homotopy theory}, volume 132 of {\em Mathematical
  Surveys and Monographs}.
\newblock American Mathematical Society, Providence, RI, 2006.

\bibitem{pereira13}
L.~A. Pereira.
\newblock A general context for {G}oodwillie {C}alculus.
\newblock arXiv:1301.2832, 2013.

\bibitem{ss03stabmodcat}
S.~Schwede and B.~Shipley.
\newblock Stable model categories are categories of modules.
\newblock {\em Topology}, 42(1):103--153, 2003.

\bibitem{weiss95}
M.~Weiss.
\newblock Orthogonal calculus.
\newblock {\em Trans. Amer. Math. Soc.}, 347(10):3743--3796, 1995.

\end{thebibliography}

\begin{tabular}{ll}
\begin{tabular}{l}
David Barnes \\
Pure Mathematics Research Centre, \\
Queen's University, \\
Belfast BT7 1NN, UK\\
\url{d.barnes@qub.ac.uk}\\
\end{tabular}
&
\begin{tabular}{l}
Rosona Eldred \\
Mathematisches Institut, \\
Universit\"{a}t M\"{u}nster\\
Einsteinstr. 62, 48149 M\"{unster}, Germany\\
\url{eldred@uni-muenster.de}
\end{tabular}
\end{tabular}

\end{document}